\documentclass[12pt, twoside, leqno]{article}
\usepackage{amsmath,amsthm}
\usepackage{amssymb,latexsym}
\usepackage{enumerate,tikz}

\newtheorem{thm}{Theorem}[section]

\newtheorem{lem}[thm]{Lemma}

\theoremstyle{definition}
\newtheorem{defin}[thm]{Definition}
\newtheorem{rem}[thm]{Remark}
\newtheorem{exa}[thm]{Example}
\newtheorem{hremark}[thm]{Historical remark}

\numberwithin{equation}{section}

\newcommand{\CaC}{{\cal C}} 

\newcommand{\ZZ}{\mathbb{Z}} 
\newcommand{\RR}{\mathbb{R}}

\begin{document}


\baselineskip=17pt


\title{Three Brouwer fixed point theorems for homeomorphisms of the plane}
\author{Lucien GUILLOU}


\date{}

\maketitle


\renewcommand{\thefootnote}{}



\renewcommand{\thefootnote}{\arabic{footnote}}
\setcounter{footnote}{0}


\begin{abstract}

We prove three theorems giving fixed points for orientation preserving homeomorphisms of the plane following forgotten results of Brouwer.

\end{abstract}

\section{Introduction}

In 1910, Brouwer \cite{Brouwer1910b, Brouwer1911} proved the following three fixed point theorems (the first one is well known as the
Cartwright-Littlewood theorem, see the
 historical remark below). In all three cases we consider an orientation preserving homeomorphism $h$ of ${\mathbf R}^2$. A continuum
 is a non empty connected compact set.

\begin{thm}
\label{A}
 Let $K$ be a non separating continuum in ${\mathbf R}^2$ such that $h(K)=K$. Then $h$ admits a fixed point in $K$.
\end{thm}

Let us recall that a set $X$ is compactly connected if given any two points in $X$, there exists a subcontinuum of $X$ which contains the two
given points.

\begin{thm}
\label{B}
 Let $F$ be a closed, compactly connected, non separating subset of ${\mathbf R}^2$ without interior such that 
   $h(F)=F$. Then $h$ admits a fixed point in $F$.
More precisely, we will prove that (if $F$ is non compact) if $h$ has no fixed point in $F$, given any neighborhood $V$ of $F$, there is a simple
closed curve with
$h$-index $+1$ inside $V$.

\end{thm}

We now consider again a non degenerated and non separating continuum $K$ in ${\mathbf R}^2$ such that $h(K)=K$. We further suppose that the circle
of 
prime ends of ${\mathbf R}^2 \setminus K$ splits into two non degenerated arcs $a_1$ and $a_2$ with the same endpoints such that $\cup_{p\in a_i}
\mbox{I}(p)=K, i=1, 2$, where $\mbox{I}(p)$ is the impression of the prime end $p$ (and therefore $\mbox{int} K=\emptyset$). Equivalently, the
end points of all accessible arcs in $a_i$ are dense in $K$, $i=1, 2$ (see section 4).

\begin{thm}
\label{C}
Suppose that the orientation preserving homeomorphism $\hat{h}$ of the circle of prime ends naturally induced by $h$ preserves $a_1$ and $a_2$ (that is
fixes the
common
end points of $a_1$ and $a_2$).
Then $h$ admits two fixed points in $K$.
\end{thm}

\begin{hremark}
 In their Annals paper of 1951, Cartwright and Littlewood \cite{Cartwright-Littlewood} proved Theorem~\ref{A} using the theory of prime ends.
This was probably the
first application of that theory to dynamical systems (though as we will see with Theorems~\ref{B} and \ref{C}, Brouwer can be considered as a
precursor on that matter too) and the result as well as the ideas of its proof have been of great importance and have
generated a large body of literature.
Nevertheless, four years later, in the same Annals, Reifenberg \cite{Reifenberg} explained a short elementary proof due to Brouwer of the same
result. Strangely enough, that
paper of Reifenberg went unnoticed and several papers have been written giving various proofs of Theorem~\ref{A} \cite{Hamilton, Brown,
Medvedev, Barge-Gillette, Matsumoto-Nakayama, Blokh-...} but without recovering the original ideas.

\end{hremark}

What we offer here is another presentation of Brouwer clever but elementary ideas with two other fixed point results he obtained in that vein and 
which are seemingly new knowledge even today. I wish to thank Alexis Marin for his comments on this paper.

\section{Index along a curve}
Given an continuous path $\alpha : [a, b] \rightarrow {\bf R}^2$ and a continuous map $f : \mbox{\rm Im}\alpha \rightarrow {\bf R}^2$ without
fixed point, we
can define the index of $f$ along $\alpha$, denoted $i(f, \alpha)$, as follows. Write $f(\alpha (t))-\alpha(t) = \vert f(\alpha (t))-\alpha(t) \vert
e^{2i\pi \theta
(t)}$ for some continuous map $\theta : [a, b] \rightarrow {\bf R}$.

\begin{defin} Set $i(f, \alpha)= \theta (b)-\theta (a)$. The following properties follow immediately from those of the covering map $t \mapsto
e^{2i\pi t}$ from $\RR$ to ${\bf S}^1$.
\begin{enumerate}
\item $i(f, \alpha)$ does not depend on the choice of the lift $\theta$ of the map $\dfrac{f \circ \alpha - \alpha}{\vert f \circ \alpha - \alpha \vert}$
from $[a, b]$ to ${\bf S}^1$.

\item $\alpha$ and $\alpha \circ \phi$, where $\phi$ is a map from $[a, b]$ to itself fixing $a$ and $b$ give rise to the same index. More
generally, if $\alpha , \beta : [a, b] \rightarrow \RR^2$ are two paths homotopic rel $\{ a, b\}$ such that $f$ has no fixed point on the image
 of the homotopy, then $i(f, \alpha)=i(f, \beta)$.

\item If $\alpha$ is a
closed curve (i.e. $\alpha(a)=\alpha(b)$), then $i(f, \alpha)$ is an integer.

\item If $\alpha^{-1}$ is defined by $\alpha^{-1} (t) = \alpha (b-t+a), a \leq t \leq b$, then $i(f, \alpha^{-1} )=-i(f, \alpha )$

\item If $g$ is an orientation preserving homeomorphism of ${\RR}^2$ and $\alpha$ a closed curve, then $i(gfg^{-1}, g(\alpha ))= i(f, \alpha )$ (for a proof of this fact,
one can use that $g$ is isotopic to $id$).

(Notice that the index along a non closed curve is not invariant by conjugation.)

\end{enumerate}

One often thinks equivalently of the vector field without zero $\xi$ on $\mbox{\rm Im}\alpha$ defined by $\xi(x) = f(\alpha (t))-\alpha(t)$ if
$x=\alpha (t)$
and one defines $i(\xi, \alpha)$ as $i(f, \alpha)$.
\end{defin}

\begin{lem}
\label{ilem1}
 $i(f, \alpha) =0$ if $\alpha$ is a simple closed curve and $f$ extend without fixed point to $\mbox{\rm int}\alpha$.
 
\end{lem}

As usual, we will denote by int$\alpha$ the bounded component of $\RR^2 \setminus \mbox{\rm Im}\alpha$ when $\alpha$ is a simple closed curve, 
but if $\alpha$ is an arc (i.e. an injective path) from $[a, b]$ to $\RR^2$, int$\alpha$ will denote $\alpha(]a, b[)$. This should cause no confusion.

\begin{proof} One can suppose that $\alpha$ is given as a map from the unit circle ${\bf S}^1 = \{e^{2i\pi t} \vert 0 \leq t \leq 1 \}$  and,
using
Schoenflies Theorem, that it extends to a map
$\phi$ from the open unit disc ${\bf D}^2$ to $\mbox{Int}\alpha$. Then the map $F : (t,u) \rightarrow \dfrac{f(\phi (ue^{2i\pi t}))-\phi (ue^{2i\pi
t})}{\vert f(\phi
(ue^{2i\pi t}))-\phi (ue^{2i\pi t})\vert }$,  $0 \leq u \leq 1$, is a well
defined homotopy which lifts to a  homotopy $\theta _u$ with $\theta _0$ a constant map; now $\theta _u(1) -\theta _u (0)$ is an integer (since
$F(0, u) = F(1, u)$) depending continuously on $u$  which is $0$ if $u=0$ and so it is $0$ also when $u=1$.
\end{proof}

\begin{lem}
\label{ilem2}
 Suppose that $\alpha$ is an arc and that we are given two maps $f$ and $g$ without fixed point on
$\mbox{\rm Im}\alpha$ such that $f(\alpha (a))=g(\alpha (a))$, $f(\alpha (b))=g(\alpha (b))$. Then
\begin{enumerate}
\item $i(f, \alpha)- i(g, \alpha)=0 $ if the images of $f$ and $g$ lie inside $\RR^2 \setminus (L \bigcup \mbox{\rm Im}\alpha)$ where $L$ is a proper 
half-line from one endpoint of $\alpha$ towards infinity such that $L \bigcap \mbox{\rm Im}\alpha$ is reduced to that endpoint.
\item $i(f, \alpha)- i(g, \alpha)=1$ if the images of
$f$ and $g$ make up a Jordan curve $\CaC$, $\mbox{\rm Im}\alpha \subset \mbox{\rm int}{\CaC}$ and the orientation of $\cal C$ induced by that
of $f$ (from $f(\alpha (a))$ to $f(\alpha (b))$) is positive.
\end{enumerate}
\end{lem}

\begin{proof}

1) Since $\RR^2 \setminus (L \bigcup \mbox{\rm Im}\alpha)$ is simply connected, there exists a homotopy $F$ between $f$ and $g$ relative to the endpoints  
inside $\RR^2 \setminus (L \bigcup \mbox{\rm Im}\alpha)$. The homotopy $(t, u) \mapsto \dfrac{F(t, u)-\alpha(t)}{\vert F(t, u)-\alpha(t) \vert}$
 lifts to a homotopy $\theta _u$ which gives the result since $\theta _u (b)$ and $\theta _u (a)$ lift constantly
 $\dfrac{f(\alpha (b)) - \alpha(b)}{\vert f(\alpha (b)) - \alpha(b)\vert }$ and $\dfrac{f(\alpha (a)) - \alpha(a)}{\vert f(\alpha (a)) -
\alpha(a)\vert }$ as $u$ varies.

2) Let $\tilde f$ denote the parametrization of $\CaC$ given by the path composition of $f$ and $g^{-1}$ (where $g^{-1}(\alpha (t)) = g(\alpha
(b-t+a))$).
Applying Schoenflies theorem, we can suppose that $\CaC$ is the circle ${\bf S}^1$. The homotopy 
$F(u, t) = \tilde f(\alpha (t))-u\alpha(t)$, for $0 \leq u \leq 1$ gives the conclusion given our orientation hypothesis.

\end{proof}

\begin{lem}
\label{ilem3}
 If ${\cal C}$ is a simple closed curve and $f$ a map of ${\cal C}$ into ${\RR}^2$ without fixed point such that $f({\cal C}) \subset
\mbox{\rm int}{\cal C} \bigcup {\cal C}$, then $i(f, {\cal C}) = 1$.
\end{lem}
\begin{proof}
 Once again we apply Schoenflies theorem to reduce the proof to the case $\CaC = {\bf S}^1$ and consider the homotopy $F(u, t) =
uf(e^{2i\pi t}) - e^{2i\pi t}$.
\end{proof}

The next Lemma is the key to the ingenious index computation of Brouwer giving the proof of Theorem~\ref{A}. A variant of this computation was
rediscovered (but unpublished !) in the eighties by Bell, see \cite{Akis, Blokh-...}.

We deal with the following situation: $h$ is an orientation preserving homeomorphism of ${\bf R}^2$, $\CaC$ is a simple closed
curve in ${\bf R}^2$ positively oriented ($\mbox{\rm int}\CaC$ lies on the left of $\CaC$) and $K$ is a $h$-invariant continuum inside $\CaC$.

We suppose there exist successive points $p_0,p_1, \ldots p_n \in \CaC$ with $p_0 = p_n$ and disjoint (except perhaps for their endpoints in $K$) 
irreducible arcs $\rho_0, \rho _1, \ldots , \rho _n$ from $p_i$ to
$K$, $0 \leq i \leq n$, $\rho_0 = \rho_n$, such that $h(\rho _i p_ip_{i+1}\rho _{i+1}) \bigcap$ $\rho _i p_ip_{i+1}\rho _{i+1} = \emptyset$.  Let
$\Omega _i$ be the
bounded region
determined by $K \bigcup \rho _i \bigcup p_ip_{i+1} \bigcup \rho _{i+1}$.

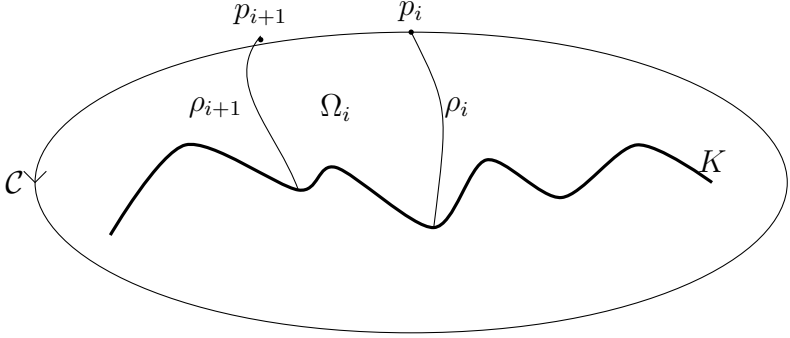
\begin{center}
\begin{tikzpicture}[scale=1] \useasboundingbox (0,0) rectangle (12,5); 
\draw  (1,2) .. controls (1,1) and (3,0) .. (6,0) .. controls (9,0) and (11,1) .. (11,2) .. controls (11,3) and (9,4) ..
(6,4) .. controls (3,4) and (1,3) .. (1,2); 
\draw (.85,2.15) -- (1,2) -- (1.15,2.15);
\draw (6,4) .. controls (6.5,3) .. (6.3,1.4) 
(1,2) node[left]{$\cal C$} (10,2) node[above]{$K$} 
(6,4) node[above]{$p_i$}  
(6.3,3) node[right]{$\rho_i$} (4,3.95) node[above]{$p_{i+1}$}  
(3.4,3) node{$\rho_{i+1}$} (5,3) node{$\Omega_i$};
\draw [very thick] plot[smooth] coordinates{(10,2) (9,2.5) (8,1.8) (7,2.3) (6.3,1.4) (5,2.2) (4.5,1.9) (3,2.5) (2,1.3)};
\draw [out=-130,in=110] (4,3.95) to (4.5,1.9); 
\fill (4,3.9) circle (1pt) (6,4) circle (1pt);
\end{tikzpicture}
\end{center}

Let $\xi$ be the vector field $\xi (x) = h(x) -x$ which is supposed to be without zero on $\CaC$. We have: 
\begin{lem}
\label{Klem}
 There exist a non zero vector field $\xi'$ on $\CaC$ with endpoints in $\mbox{int}\CaC \bigcup \CaC$ such that $i(\xi, \CaC)- i(\xi',
\CaC) = k \in {\bf N}$. And so $i(\xi, \CaC) = 1 + k$. That integer $k$ is given by the number of $i$ such that $\Omega _i \subset h(\Omega
_i)$.
\end{lem}
\begin{proof} The second sentence of the Lemma follows from the fact that $i(\xi', \CaC)=1$ since the endpoint of $\xi'\in \hbox{int}\CaC \bigcup
\CaC$
(see Lemma~\ref{ilem2}).

On $\rho _i$ described from $p_i$ towards $K$ define $q_i$ as $p_i$ if $h^{-1}(\CaC) \bigcap \rho_i= \emptyset$ or the last point
of $h^{-1}(\CaC) \bigcap \rho_i$ if this set is not empty. Let $\phi _i$ be an orientation preserving homeomorphism from $p_ip_{i+1}$ to
$q_ip_ip_{i+1}q_{i+1}$ and let $\tilde \xi$ be the vector field along $\CaC$ defined by $\tilde \xi (x) = h(\phi _i(x))-x$ for $x \in
p_ip_{i+1}$ (this
is non zero since $h(\rho_i p_i p_{i+1}\rho_{i+1}) \bigcap p_i p_{i+1} = \emptyset$).

One has obviously $i(\xi , \CaC) = \sum _{i=0}^{n-1}i(\xi, p_ip_{i+1})$ and also $i(\xi , \CaC) = \sum _{i=0}^{n-1}i(\tilde \xi, p_ip_{i+1})$.
 Indeed, the path composition of the maps $\phi_i$ and $\phi _{i+1}$ from $p_ip_{i+1}p_{i+2}$ to $q_ip_ip_{i+1}q_{i+1}p_{i+1}p_{i+2}q_{i+2}$ is 
homotopic rel endpoints to a map onto the arc $q_ip_ip_{i+1}p_{i+2}q_{i+2}$. Combining all the $\phi _i$ we get that $\sum_{i=0}^{n-1} i(\tilde
\xi
, p_ip_{i+1}) = i(\hat \xi, \CaC)$ where $\hat \xi = h(\phi (x))-x$ and $\phi$ is an orientation preserving homeomorphism of $\CaC$ fixing $p_0$
and therefore homotopic
to the identity rel $p_0$, so that $i(\hat \xi, \CaC) = i(\xi, \CaC)$.

We now distinguish three cases.

1) $h(\Omega _i) \bigcap \Omega _i = \emptyset$.

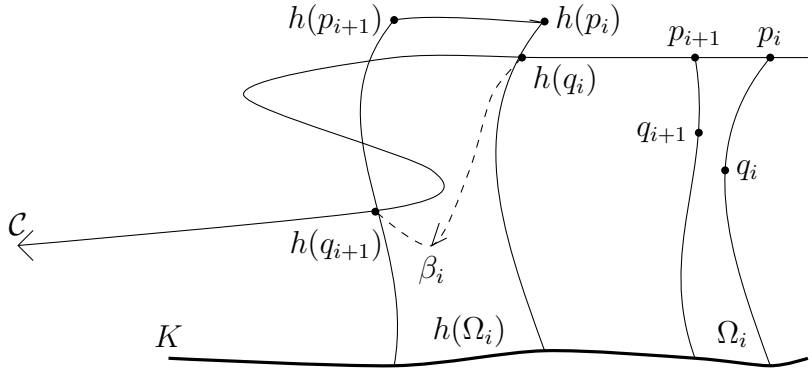
\begin{center}
\begin{tikzpicture}[scale=1] \useasboundingbox (0,0) rectangle (12,5); 
\draw plot[smooth] coordinates{(1,2) (6,2.5) (6.5,3) (4,4) (6,4.5) (8,4.5) (10,4.5) (11,4.5) (11.5,4.5)}; 
\draw [very thick] plot[smooth] coordinates{(11.5,.5) (11,.4) (10,.5) (8,.6) (6,.4) (3,.5)};
\draw [out=-130,in=110] (11,4.5) to (11,.4); 
\draw [out=-80,in=110] (10,4.5) to (10,.5);
\draw [out=-130,in=110] (8,5) to (8,.6);
\draw [out=-130,in=80] (6,5) to (6,.4);
\draw [out=10,in=-10] (6,5) to (7.8,5);
\draw [dashed] plot[smooth] coordinates{(7.7,4.5) (7.3,4) (7,3) (6.5,2) (5.7,2.5)};
\draw (10.5,.5) node[above]{$\Omega_i$};
\draw (7,.5) node[above]{$h(\Omega_i)$};
\draw (1,2) node[above]{$\cal C$};
\draw (3,.5) node[above]{$K$};
\draw (11,4.5) node[above]{$p_i$} (10,4.5) node[above]{$p_{i+1}$};
\fill (11,4.5) circle (1.5pt) (10,4.5) circle (1.5pt);
\draw (6.5,2) node[below]{$\beta_i$};
\draw (6,2) node[left]{$h(q_{i+1})$};
\fill (5.75,2.45) circle (1.5pt);
\draw (7.7,4.2) node[right]{$h(q_i)$};
\fill (7.7,4.5) circle (1.5pt);
\draw (6,5) node[left]{$h(p_{i+1})$} (8,5) node[right]{$h(p_i)$};
\fill (6,5) circle (1.5pt) (8,4.98) circle (1.5pt);
\fill (10.4,3) circle (1.5pt) (10.05,3.5) circle (1.5pt);
\draw (10.4,3) node[right]{$q_i$} (10.05,3.5) node[left]{$q_{i+1}$};
\draw (1.2,1.8)--(1,2)--(1.2,2.2);
\draw (6.7,2.1)--(6.5,2)--(6.6,2.3);
\end{tikzpicture}
\end{center}

Choose any arc $\beta _i$ from $h(q_i)$ to $h(q_{i+1})$ inside $\CaC$ except for its endpoints. Let $\lambda$ and $\mu$ be parametrisations by
$[0, 1]$ of
the arcs $p_ip_{i+1}$ and $\beta _i$ respectively and define $\xi '$ on $p_ip_{i+1}$ by $\xi ' (\lambda (t)) = \mu (t) - \lambda (t)$. One has
$i(\tilde
\xi, p_ip_{i+1}) = i(\xi ', p_ip_{i+1})$ by Lemma~\ref{ilem2}(1) since the arc $p_ip_{i+1}$ lies outside the Jordan curve made of $\beta _i$ and
$h(q_ip_ip_{i+1}q_{i+1})$.

2) $h(\Omega _i) \subset \Omega_i$.

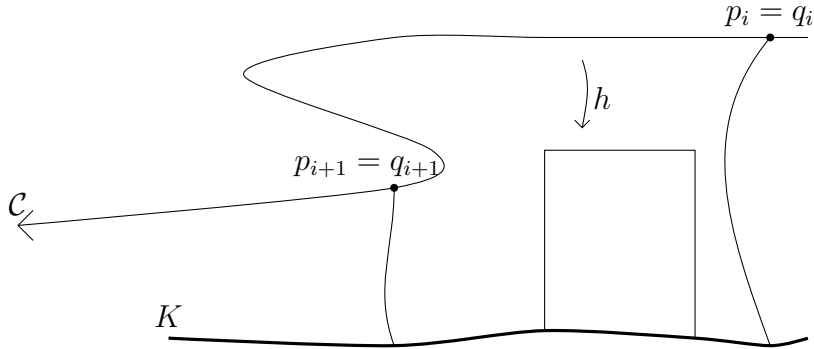
\begin{center}
\begin{tikzpicture}[scale=1] \useasboundingbox (0,0) rectangle (12,5); 
\draw plot[smooth] coordinates{(1,2) (6,2.5) (6.5,3) (4,4) (6,4.5) (8,4.5) (10,4.5) (11,4.5) (11.5,4.5)}; 
\draw [very thick] plot[smooth] coordinates{(11.5,.5) (11,.4) (10,.5) (8,.6) (6,.4) (3,.5)};
\draw [out=-130,in=110] (11,4.5) to (11,.4);
\draw [out=-90,in=110] (6,2.5) to (6,.4);
\draw (1,2) node[above]{$\cal C$};
\draw (3,.5) node[above]{$K$};
\draw (11,4.5) node[above]{$p_i=q_i$};
\fill (11,4.5) circle (1.5pt);
\draw (5.65,2.5) node[above]{$p_{i+1}=q_{i+1}$};
\fill (6,2.5) circle (1.5pt);
\draw (1.2,1.8)--(1,2)--(1.2,2.2);
\draw (10,.5)--(10,3)--(8,3)--(8,.6);
\draw [out=-70,in=80] (8.5,4.2) to (8.5,3.3);
\draw (8.5,3.7) node[right]{$h$};
\draw (8.4,3.4)--(8.5,3.3)--(8.6,3.4);
\end{tikzpicture}
\end{center}

In that case $q_i = p_i$ and $q_{i+1} = p_{i+1}$ and we let $\xi ' = \tilde \xi$ along $p_i p_{i+1}$. Obviously, $i(\tilde \xi, p_ip_{i+1}) =
i(\xi ', p_ip_{i+1})$

3) $\Omega _i \subset h(\Omega _i)$.

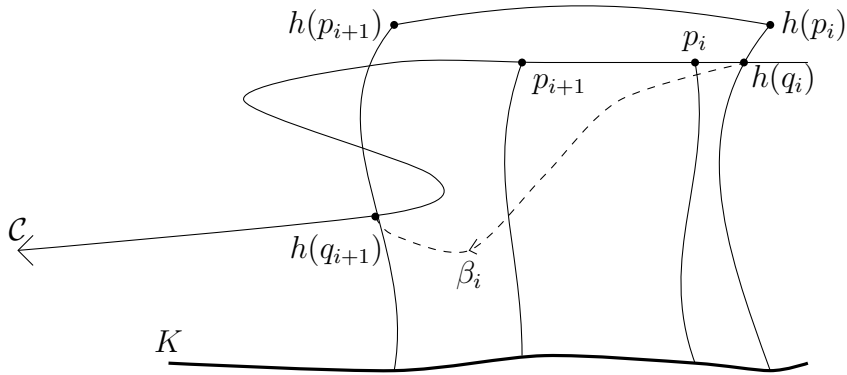
\begin{center}
\begin{tikzpicture}[scale=1] \useasboundingbox (0,0) rectangle (12,5); 
\draw plot[smooth] coordinates{(1,2) (6,2.5) (6.5,3) (4,4) (6,4.5) (8,4.5) (10,4.5) (11,4.5) (11.5,4.5)}; 
\draw [very thick] plot[smooth] coordinates{(11.5,.5) (11,.4) (10,.5) (8,.6) (6,.4) (3,.5)};
\draw [out=-130,in=110] (11,5) to (11,.4); 
\draw [out=-80,in=110] (10,4.5) to (10,.5);
\draw [out=-110, in=90] (7.7,4.5) to (7.7,.6);
\draw [out=-130,in=80] (6,5) to (6,.4);
\draw [out=10,in=170] (6,5) to (11,5);
\draw [dashed] plot[smooth] coordinates{(10.69,4.5) (9,4) (8,3) (7,2) (6,2.2) (5.75,2.45)};
\draw (7,2) node[below]{$\beta_i$};
\draw (7.1,1.95)--(7,2)--(7.1,2.15);
\draw (1,2) node[above]{$\cal C$};
\draw (3,.5) node[above]{$K$};
\draw (10,4.5) node[above]{$p_i$} (10.6,4.28) node[right]{$h(q_i)$};
\fill (10.65,4.5) circle (1.5pt) (10,4.5) circle (1.5pt);
\draw (6,2) node[left]{$h(q_{i+1})$};
\fill (5.75,2.45) circle (1.5pt);
\draw (7.7,4.2) node[right]{$p_{i+1}$};
\fill (7.7,4.5) circle (1.5pt);
\draw (6,5) node[left]{$h(p_{i+1})$} (11,5) node[right]{$h(p_i)$};
\fill (6,5) circle (1.5pt) (11,5) circle (1.5pt);
\draw (1.2,1.8)--(1,2)--(1.2,2.2);
\end{tikzpicture}
\end{center}

We define $\xi'$ as in case 1). We have $i(\tilde \xi , p_ip_{i+1}) - i(\xi ', p_ip_{i+1}) = 1$ since  the arc $p_ip_{i+1}$ is contained inside
the Jordan
curve made of $ \beta_i$ and $h(q_ip_ip_{i+1}q_{i+1})$ and $h$ is orientation preserving so that $h(q_ip_ip_{i+1}q_{i+1})$ is oriented from
$h(q_i)$ to $h(q_{i+1})$, see Lemma~\ref{ilem2}.

Since $h(\rho _ip_ip_{i+1}\rho _{i+1}) \bigcap \rho _ip_ip_{i+1}\rho _{i+1} = \emptyset$ these three cases exhaust all possibilities and this
concludes the
proof of the Lemma.
\end{proof}
\bigskip

\section{Proof of Theorem~\ref{A}}

The next Lemma is very classical \cite{Schoenflies} and is a key step to show that ${\RR}^2 \setminus K$ is homeomorphic to ${\RR}^2 \setminus \{
0
\}$ if and only if $K$ is a non separating continuum in ${\RR}^2$ (which is not really needed or used here but explain the Brouwer terminology ``circular
continum'': ${\RR}^2 \setminus K$ looks like a circular region).

We will follow a presentation of Sieklucki \cite{Sieklucki}.

\begin{lem}
\label{dlem}
 Let $K \subset \RR^2$ a non empty non separating continuum. Then there exists a sequence $B_n$,$n \geq 0$, of topological closed discs such that
\begin{enumerate}
\item $K  = \bigcap B_n$
\item $B_{n+1} \subset B_n$
\item For every $b \in \mbox{\rm Fr}B_n$ there exists a rectilinear segment $\rho = \rho (b)$ from $b$ to some point $x(b) \in \mbox{\rm Fr}K$
such that
$\rho (b) \setminus \{b, x(b) \} \subset \mbox{\rm Int}B_n \setminus K$, $\mbox{\rm diam}(\rho (b))< \sqrt{2}.2^{-n}$ and $\rho (b) \bigcap \rho
(b')$ is empty or
reduced to $x(b) = x(b')$ ($b \not = b'$).
\end{enumerate}

\end{lem}

\begin{proof} For $n\geq 1$, we consider the tiling $C_n$ of the plane by the closed squares centered at $(\dfrac{k}{2^n}, \dfrac{l}{2^n})$ of
side length $\dfrac{1}{2^n}$ ($k, l \in \ZZ$). Let $Q_n$ be the union of all squares of $C_n$ which meet $K$ and $B_n$ be the union of $Q_n$ and
all bounded components of ${\RR}^2 \setminus Q_n$. Then $B_n$ is a topological disc and $B_{n+1} \subset B_n$ (since $\mbox{\rm Fr}Q_{n+1}
\subset Q_n$).

To show that $K = \bigcap _{n>0}B_n$, let $x \in {\RR}^2 \setminus K$ and use that $K$ is non separating to find a half line $l$ from $x$ to $\infty$ 
such that $l \bigcap K
=\emptyset$. Then $\mbox{\rm d}(l, K)>0$ and if $\dfrac{\sqrt{2}}{2^n} < \mbox{\rm d}(l, K)$ then $x \notin B_n$ (otherwise there exists $y \in
l\bigcap \mbox{\rm Fr}B_n$ and $\mbox{\rm d}(y, K)< \dfrac{\sqrt{2}}{2^n} < \mbox{\rm d}(l, K)$ which is absurd).

As for 3), given $b \in \mbox{\rm Fr}B_n$, then  $b$ belongs to a side of some square $Q$ of $C_n$. Let $x(b)$ be some nearest point of $K$ in
$Q$ and $\rho (b)$ be the rectilinear segment from $b$ to $x(b)$. Given $b, b' \in \mbox{\rm Fr}B_n$, if $\rho(b)$ and $\rho(b')$ do not belong to the
same square, clearly $\rho (b) \bigcap \rho (b')$ is empty or reduced to $x(b)=x(b')$. If $\rho(b)$ and $\rho(b')$ belong to the same square $Q$, 
 we conclude with the following elementary Lemma.
\end{proof}

\begin{lem}
 \label{Qlem}
Let $Q$ is a square and $K$ a closed subset in $Q$. For $b \in \mbox{\rm Fr}Q$, let $x(b)$ be a nearest point of $K$ in
$Q$ and $\rho (b)$ be the rectilinear segment from $b$ to $x(b)$. Then $\rho (b) \bigcap \rho(b')$ is empty or
reduced to $x(b) = x(b')$ if $b \not = b'$.
\end{lem}

\begin{proof}
 Suppose there exists a point $c$ in $\rho (b) \bigcap \rho (b')$. The
inequalities
\[ \vert b-c\vert + \vert c-x(b) \vert = \vert b-x(b) \vert \leq \vert b-x(b') \vert \leq \vert b-c \vert + \vert c-x(b') \vert \]
 give $\vert c-x(b) \vert \leq \vert c-x(b') \vert$ and by symmetry  $\vert c-x(b) \vert = \vert c-x(b') \vert$. Therefore $\vert b-x(b') \vert
\geq \vert c-b \vert + \vert c-x(b) \vert = \vert c-b \vert + \vert c-x(b') \vert$ so that $\vert b-x(b') \vert = \vert b-c \vert + \vert c-x(b')
\vert$ and, by symmetry, $\vert b'-x(b) \vert = \vert b'-c \vert + \vert c-x(b) \vert$. This implies, if $b \not = b'$, that $c = x(b) = x(b')$.
 
\end{proof}

We can now complete the proof of Theorem~\ref{A}.
We have an orientation preserving homeomorphism $h$ of ${\bf R}^2$ and a non separating continuum $K$ in ${\bf R}^2$ such that $h(K)=K$. We have
to prove that $h$ admits a fixed point in $K$.

\begin{proof} Let us suppose that $h$ has no fixed point in $K$. Then we can find a neighborhood $V$ of $K$ and an $\epsilon >0$ such that
dist$(h(x), x)
> 3\epsilon$ on $V$. According to Lemma~\ref{dlem}, one can find a Jordan curve $\CaC $  contained in $V$  and the $\epsilon$-neighborhood of $K$
such that
$K
\subset \mbox{int}\CaC$, successive points $p_0,p_1, \ldots, p_n=p_0$ on $\CaC$ and disjoint arcs (except perhaps for their extremities on $K$)
$\rho _0,
\ldots ,
\rho _{n-1}$ from $p_i$ to $K$ such that diam$p_ip_{i+1}$ and diam$\rho _i$ are less than $\epsilon$ and consequentely $h(\rho _i p_ip_{i+1}\rho
_{i+1})
\bigcap \rho _i p_ip_{i+1}\rho _{i+1} = \emptyset$. We are then in position to apply Lemma~\ref{Klem} which implies that $ i(\xi , {\CaC}) >
0 $ in
contradiction to Lemma~\ref{ilem1}.

\end{proof}

\begin{rem} Another fixed point theorem of Brouwer \cite{Brouwer1912} is as follows.
\begin{thm}
\label{D}
 Let $h$ be an orientation preserving homeomorphism of ${\RR}^2$ and let $K$ be a non empty compact subset of ${\RR}^2$ such that $h(K)=K$
then $h$ admits a fixed point (in ${\RR}^2$).
\end{thm}
This result follows quickly (using a perturbation argument for example) from the fact that an orientation preserving homeomorphism of ${\RR}^2$ 
without fixed point has no periodic point which can also be proved by an index computation \cite{Brouwer1912, Guillou1994}.

Notice that the ``short'' proof of Theorem~\ref{A} by Hamilton \cite{Hamilton} and the ``short short'' proof by Brown \cite{Brown} are indeed
merely reduction of
Theorem~\ref{A} to the above result. In fact, if $h$ had no fixed point inside the non separating continuum $K$, these authors construct (by bare
hands for Hamilton, by a simple covering argument for Brown) another extension $h'$ to ${\RR}^2$ of the restriction of $h$ to $K$ which is
orientation preserving and without any fixed point: this is a contradiction to Theorem~\ref{D}.
 
\end{rem}

\section{Prime ends}
We give only a brief sketch of the theory of prime ends in a pre-Caratheodory style, based on the notion of accessible arc and the cyclic order
that can be given to equivalence classes of such arcs and as used by Brouwer. See  \cite{Milnor} or \cite{Pommerenke} for modern and
more complete expositions.

We consider a non empty continuum $K \subset {\bf S}^2 = {\RR}^2 \bigcup \{ \infty \}$ such that $U={\bf S}^2 \setminus K$ is non empty and connected.

\begin{defin}
 A point $x \in \mbox{\rm Fr}U=\mbox{Fr}K$ is said {\it accessible} (from $U$) if there is an arc $\gamma : [a, b] \rightarrow \overline{U}$ 
such
that $\gamma ([a, b)) \subset U$ and $\gamma (b)=x$. The point $x$ is the endpoint of $\gamma$ and $\gamma$ is an {\it access arc} to $x$.

We define an equivalence relation on the set of access arcs from $x_0$ to $\mbox{\rm Fr}U$ by $\gamma \sim \gamma '$ if $\gamma$ and $\gamma'$ 
have the same endpoint $x \in \mbox{Fr}U$ and $\gamma$ is isotopic
 to $\gamma'$ in $U \bigcup \{ x\}$ rel $\{ x\}$. 
Notice that the end point of the class $p=[ \gamma ]$ is well defined. 

\end{defin}

Some basic facts are:

 -The set of accessible points is dense in $\mbox{\rm Fr}U=\mbox{Fr}K$.
 
 -Given a finite number of distinct equivalence classes of access arcs $p_1, p_2, \ldots , p_n$,
 one can find disjoint access arcs $\gamma _1, \ldots , \gamma _n$ where $\gamma_i \in p_i$. 
 
 -Using a circle surrounding $K$ and meeting $\gamma _1, \ldots , \gamma _n$ (see Lemma~\ref{dlem}) we can transfer a cyclic order on this circle to the set 
$\{ p_1, p_2, \ldots ,p_n\}$ and thus define a cyclic order on the set of equivalence classes of access arcs (given coherent choices of orientation for the circles 
surrounding $K$; we will assume that $K$ is to the left of each such circle).  We can therefore
talk of the closed interval $[p, p']$ given two equivalence classes $p$ and $p'$.

-Given two distinct equivalence classes $p$ and $p'$, there exists a third one $p''$ such that $p<p''<p'$.

-Given an equivalence class $p$, there exist sequences of equivalence classes $(p_n)_{n\geq 0}$ and $(p'_n)_{n\geq 0}$ such that $\bigcap_{n\geq 0}[p_n , p'_n ]=\{ p\}$.

\begin{defin}
We now consider sequences of decreasing intervals $[p_n, p_n']$ such that $\bigcap _{n\geq 0} [p_n, p_n']$
 is empty or reduced to one point, where $p_n$ and $p_n'$ are sequences of equivalence classes of access arcs. Two such sequences $[p_n,
p_n']$ and
$[q_n, q_n']$ are considered equivalent if for each $n$ there exist $r$ such that $[p_n, p_n'] \supset [q_r,
q_r']$ and $s$ such that $[q_n, q_n'] \supset [p_s, p_s'$]. Equivalence classes of such sequences of intervals
 define the {\it prime ends}.
\end{defin}

Given the last fact above, equivalence classes of access arcs are naturally seen as prime ends.

The cyclic order on the equivalence classes of access arcs can be extended to the set of all prime ends and a classical result of the theory of ordered sets 
gives a cyclic order preserving
 bijection of the set of prime ends to the circle. If
we give the order topology to the set of prime ends such a bijection becomes a homeomophism. Also, any homeomorphism $h$ of $\overline{U}$ extend
to the set of equivalence classes of access arcs (by $h([\gamma ]) = [h(\gamma )]$) and so to the circle of prime ends.

\begin{defin}
 A prime end $p$ being defined by a sequence $[p_n, p_n']$, we define its {\it impression} as the set of all points of $\mbox{\rm
Fr}\overline{U}$ which are limits of a sequence of end points of access arcs $\beta _k$ such that $[\beta _k] \in [p_{n_k},
p_{n_k}']$ for some increasing subsequence $n_k$ of the integers.
\end{defin} 
This impression, denoted I($p$), does not depend on the choice of the
sequences $p_n$ and $p_n'$ and can be shown to be a subcontinuum of $\mbox{\rm Fr}U \subset {\bf S}^2$. The union over all prime ends of these
impressions form a covering of $\mbox{\rm Fr}U$ but notice that
 different prime ends may have the same impression and it is even possible that some impressions are equal to $\mbox{\rm Fr}U$ (see the
indecomposable continuum of Brouwer \cite{Brouwer1910, Rutt, Rogers}).

\begin{defin}
 A {\it cut} $c$ of $U$ is an arc $c : [a, b] \rightarrow \overline{U} \setminus \{ x_0 \}$ such that $c(a), c(b) \in \mbox{\rm Fr}U$ and $c(a, b)
\subset U$.

As is well known, a cut separate $U$ into exactly two regions and we will call the region not containing $\infty$ the bounded region determined by $c$ (and Fr$U$).
\end{defin}

\section{Proof of Theorem~\ref{B}}

We begin with some preliminary lemmas.

\begin{lem}
\label{Tlem}

 Let a simple closed curve ${\cal C}$ be composed of three consecutive arcs $\alpha$, $\beta$, $\gamma$ $: [0, 1] \rightarrow \RR^2$ with disjoint
interiors and $h : {\cal C}
\rightarrow {\RR}^2$ a map without fixed point. Suppose that
\begin{enumerate} 
\item $h(\alpha) \subset \mbox{\rm int}{\cal C} \bigcup {\cal C}$ and that $h(\alpha(0)) \in \mbox{\rm int}\alpha$.
\item $h(\beta(0)) = \beta(1)$ and $h(\beta) \bigcap \beta = \beta(1)$.
\item $\beta \setminus \beta(1)$ lies in the unbounded region of $\RR^2 \setminus h({\cal C})$. 
\item $h(\gamma) \bigcap \gamma = \emptyset$. 
\end{enumerate}

Then $i(h, {\cal C})= 1$.

\end{lem}

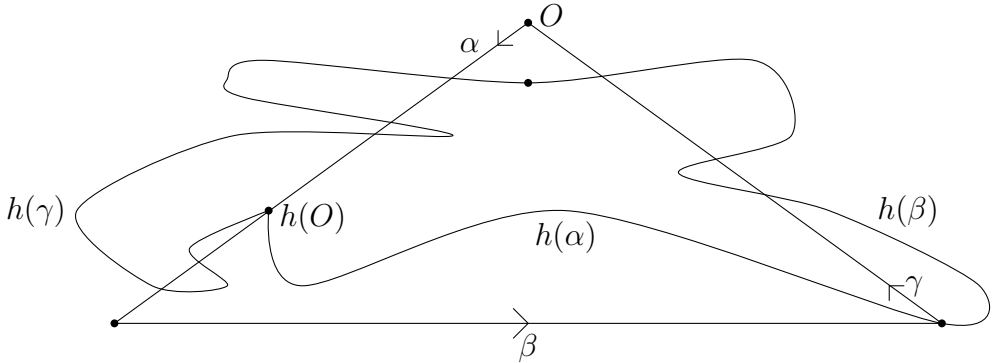
\begin{center}
\begin{tikzpicture}[scale=1] \useasboundingbox (0,0) rectangle (12,5); 
\draw plot[smooth] coordinates{(2.55,2) (3,1) (6.5,2) (11.5,0.5) (12,1) (10,2) (8,2.5) (9.5,3) (9,4) (6,3.7) (2.5,4) (2,3.75) (2.25,3.5) (5,3) 
(2.1,3) (0,2) (1,1) ((2,1) (1.5,1.5) (2.55,2)}; 
\draw (0.5,0.5)--(11.5,0.5)--(6,4.5)--(0.5,0.5);
\fill (6,4.5) circle (1.5pt);
\fill (.5,.5) circle (1.5pt);
\fill (11.5,.5) circle (1.5pt);
\fill (2.55,2) circle (1.5pt);
\fill (6,3.7) circle (1.5pt);
\draw (6,4.6) node[right]{$O$};
\draw (2.55,1.9) node[right]{$h(O)$};
\draw (0,2) node[left]{$h(\gamma)$};
\draw (6,0.5) node[below]{$\beta$};
\draw (10.5,2) node[right]{$h(\beta)$};
\draw (5.5,4.2) node[left]{$\alpha$};
\draw (11.4,1) node[left]{$\gamma$};
\draw (6.5,2) node[below]{$h(\alpha)$};
\draw (5.8,0.3)--(6,0.5)--(5.8,0.7);
\draw (5.6,4.4)--(5.6,4.2)--(5.8,4.2);
\draw (10.8,0.8)--(10.8,1)--(11,1);
\end{tikzpicture}
\end{center}

\begin{proof} Let $\star$ denote the path composition and $O$ be $\alpha(0)$. By hypothesis, $\beta \setminus \beta(1)$ lies in the unbounded
complementary region of the closed curves 
 $h({\cal C})$ and $h(\alpha)\star \gamma\star Oh(O)$ (where $Oh(O)$ is a subarc of $\alpha$). Inside $\RR^2 \setminus (\beta
\setminus \beta(1))$, we have $h(\beta)$ homotopic rel endpoints to $h(\alpha)^{-1}
\star h(\gamma)^{-1}$ and $h(\alpha)^{-1}$ homotopic rel endpoints to $\gamma \star Oh(O)$.
Therefore, according to Lemma~\ref{ilem2}, we can replace the field $\xi = h(x) - x$ on $\cal C$ by a field $\xi' = h'(x)-x$, (where $h' : {\cal
C}
\rightarrow \RR^2$ is without fixed point and equal to $h$ on $\gamma$),  with the same index as $\xi$ and whose endpoint describes $\gamma
\star Oh(O)\star h(\gamma)^{-1}\star h(\gamma)$ as
its origin describes $\beta \star \gamma$. The natural homotopy of $h(\gamma)^{-1}\star h(\gamma)$ (supported by $\mbox{\rm Im}(h\circ \gamma)$) to the
constant map on $h(O)$ does not meet fixed points of $h'$
since $\beta \star \gamma$ does not meet $h'(\gamma) = h(\gamma)$. Finally we get a new field on ${\cal C}$ with the same index as $\xi$ whose
endpoint describes
$h(\alpha)\star \gamma
\star Oh(O)$, that is a curve
inside int$\cal C \bigcup \cal C$ and we conclude with Lemma~\ref{ilem3}.
\end{proof}

We will need a variation on Lemma~\ref{Klem}.

We consider the following situation:
 $\cal C$ is a simple closed curve positively oriented, $K$ a closed connected set such that $K \bigcap \mbox{\rm int}{\cal C} \not = \emptyset$ and
$h(K) = K$, and $\xi$ is 
a vector field without zero along $\cal C$  . Suppose there is an arc $\alpha \subset {\cal C}$ with
$\mbox{\rm int}\alpha \bigcap K =\emptyset$ and
successive
points $p_0, p_1, \cdots , p_n \in \alpha$ (where $p_0$ is the origin of $\alpha$ and $p_n$ the endpoint of $\alpha$) such that $h(p_0), h(p_n)
\in
\mbox{\rm int}{\cal C} \bigcup {\cal C}$ and arcs $\rho_1, \rho_2, \cdots, \rho_{n-1}$ where $\rho_i$ joins $p_i$ to $K$ irreductibly which are
disjoint except perhaps for their
endpoint in $K$. Let $\Omega_i$ be the bounded region determined by $\rho_i p_ip_{i+1}\rho_{i+1} $ and $K$ ($1 \leq i \leq n-1$).

\begin{lem}
\label{Kbislem}
If the preceding data satisfy
\begin{enumerate}
 \item $\rho_i \subset \mbox{\rm int}{\cal C} \bigcup \{ p_i \}$.
\item $h(\rho_i p_ip_{i+1}\rho_{i+1}) \bigcap \rho_i p_ip_{i+1}\rho_{i+1} = \emptyset$, $1 \leq i \leq n-2$.
\item $h(\rho_i) \bigcap \mbox{\rm int}{\cal C} \not = \emptyset$.
\item Either $p_0$ and $h(p_0)$ belong to $\CaC$ and $h(p_0 p_1) \bigcup h(\rho_1)$ does not separate $p_0 p_1$ from infinity in $\RR^2
\setminus \mbox{\rm int}{\cal C}$ or $p_0$ (and $h(p_0)$) belong to $K$ and the bounded region $\Omega_0$ determined by $p_0p_1 \bigcup \rho_1$
and $K$ satisfies $\Omega_0 \bigcap h(\Omega_0) = \emptyset$. And similarly for $p_n$.
\end{enumerate}
Then there exists another vector field without zero $\xi'$ along $\cal C$ equal to $\xi$ outside $\alpha$ such that the endpoints of $\xi'$ belong
to
$\mbox{\rm int}{\cal C} \bigcup {\cal C}$ along $\alpha$ and which satisfies $i(\xi, \alpha)- i(\xi', \alpha) = k$ where $k\geq 0$ is given by
the number
of $i$ such that $\Omega_i \subset h(\Omega_i)$ ($1 \leq i \leq n-1$).
\end{lem}

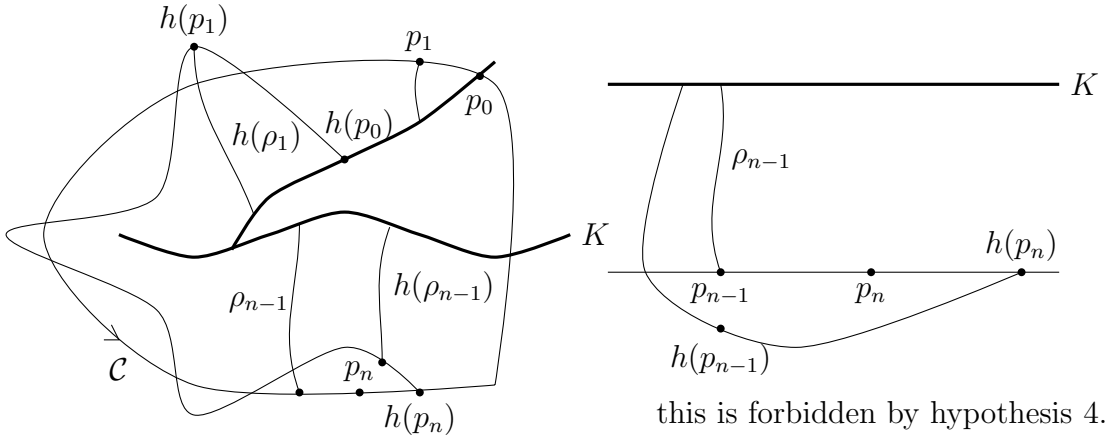
\begin{center}
\begin{tikzpicture}[scale=1] \useasboundingbox (0,0) rectangle (15,5); 
\draw [very thick] plot[smooth] coordinates{(2,2.5) (3,2.2) (4,2.5) (5,2.8) (6,2.5) (7,2.2) (8,2.5) }; 
\draw [very thick] plot[smooth] coordinates{(3.5,2.3) (4,3) (5,3.5) (6,4) (7,4.8)};
\draw (8,2.5) node[right]{$K$};
\draw plot[smooth] coordinates{(7,.5) (7,4.5) (3,4.5) (1,2.5) (3,.5) (7,.5)};
\draw plot[smooth] coordinates{(5,3.5) (3,5) (2.5,3) (.5,2.5) (2.5,1.5) (3,.1) (5,1) (6,.4)};
\draw (2,1) node[below]{$\cal C$};
\draw (1.9,1.3)--(2,1.1)--(1.8,1.1);
\fill (5,3.5) circle (1.5pt);
\draw (5.2,3.6) node[above]{$h(p_0)$};
\fill (6.8,4.61) circle (1.5pt);
\draw (6.8,4.5) node[below]{$p_0$};
\draw [out=-110,in=100] (6,4.8) to (6,4);
\fill (6,4.8) circle (1.5pt);
\draw (6,4.8) node[above]{$p_1$};
\fill (6,.4) circle (1.5pt);
\draw (6,.4) node[below]{$h(p_n)$};
\draw [out=-90,in=110] (3,5) to (3.8,2.75);
\draw (3.35,3.8) node[right]{$h(\rho_1)$};
\fill (3,5) circle (1.5pt);
\draw (3,5) node[above]{$h(p_1)$};
\draw [out=90,in=-110] (5.5,.8) to (5.6,2.6);
\draw (5.5,1.8) node[right]{$h(\rho_{n-1})$};
\draw [out=110,in=-90] (4.4,.4) to (4.4,2.65);
\draw (4.4,1.6) node[left]{$\rho_{n-1}$};
\fill (5.5,.8) circle (1.5pt);
\fill (4.4,.4) circle (1.5pt);
\fill (5.2,.4) circle (1.5pt);
\draw (5.2,.4) node[above]{$p_n$};
\draw [very thick] (8.5,4.5)--(14.5,4.5);
\draw (14.5,4.5) node[right]{$K$};
\draw (8.5,2)--(14.5,2);
\draw plot[smooth] coordinates{(14,2) (11,1) (9,2) (9.5,4.5)};
\fill (14,2) circle (1.5pt);
\draw (14,2) node[above]{$h(p_n)$};
\draw [out=110,in=-80] (10,2) to (10,4.5);
\fill (10,2) circle (1.5pt);
\draw (10,2) node[below]{$p_{n-1}$};
\draw (10,3.5) node[right]{$\rho_{n-1}$};
\fill (12,2) circle (1.5pt);
\draw (12,2) node[below]{$p_n$};
\fill (10,1.25) circle (1.5pt);
\draw (10,1.2) node[below]{$h(p_{n-1})$};
\draw (9,.1) node[right]{this is forbidden by hypothesis 4.};

\end{tikzpicture}
\end{center}

\begin{proof}
The proof is very similar to the proof of Lemma~\ref{Klem}. We define again, on $\rho_i$ described from $p_i$ towards $K$, $q_i$ as $p_i$ if
$h^{-1}({\cal C}) \bigcap \rho_i = \emptyset$ or the last point of $h^{-1}({\cal C}) \bigcap \rho_i$ if this set is not empty. Let $q_0 =
p_0$ and $q_n = p_n$ and proceed now exactly as in the proof of Lemma~\ref{Klem}, the arcs $\beta_i$ 
such that $\beta_i \subset \mbox{int}{\cal C} \bigcup {\cal C}$  existing trivially using 3), $\mbox{int}{\cal C} \bigcup
{\cal C}$ being arc connected. Given hypothesis 4., the contributions of $\xi$ and $\xi '$ to their index are equal on $p_0p_1$ and $p_{n-1}p_n$,
whence the last formula.

\end{proof}

We will also need the following slight extension of Lemma~\ref{dlem}.

We consider a non empty non separating continuum in the plane and a finite number of pairwise disjoints arcs  $\gamma _i :
[-1, 0] \rightarrow \RR^2$ , $1 \leq i \leq k$, such that $\gamma_i(0) \in K$ and $\gamma_i([-1,
0))$  lies in $\RR^2 \setminus K$. 

\begin{lem}
 \label{dbislem}
There exists a sequence $B_n$, $n \geq 0$, of topological closed discs such that
\begin{enumerate}
 \item $K = \bigcap B_n$
\item $B_{n+1} \subset B_n$
\item For every $b \in \mbox{\rm Fr}B_n \setminus \bigcup \mbox{\rm Im}\gamma_i$ there exists an arc $\rho = \rho(b)$ from $b$ to some point
$x(b) \in \mbox{\rm Fr}K$ such that $\rho(b) \setminus \{ b, x(b) \} \subset \mbox{\rm Int}B_n \setminus (K \bigcup (\bigcup_i \mbox{\rm
Im}\gamma_i$)), and $\rho(b) \bigcap \rho(b')$ is at most a point in $K$ if $b \not = b'$.
\item diam$\rho(b) \rightarrow 0$ uniformly in $b \in \mbox{\rm Fr}B_n$ as $n \rightarrow +\infty$.
\end{enumerate}

\begin{proof}
 Schoenflies theorem gives us a homeomorphism $\phi$ of $\RR^2$ such that each $\phi ^{-1}(\gamma_i)$ is a vertical segment 
 with abscissa an integer. We now apply the proof of Lemma~\ref{dlem} to the continuum $\phi^{-1}(K)$, considering the tiling $C_n$ of the
plane by the closed squares of center $(\dfrac{k}{2^n}, \dfrac{l}{2^n})$ and side length $\dfrac{1}{2^n}$, $k, l \in \ZZ$. Since the $\gamma_i$
are contained in the $1$-skeleton of $C_n$, we get the desired result for the continuum $\phi^{-1}(K)$ and the $\phi^{-1}(\gamma_i)$. We conclude
using the uniform continuity of $\phi$ on any big ball containing the whole sequence of discs associated to $\phi^{-1}(K)$.
\end{proof}
\end{lem}

\begin{lem}
 A closed compactly connected subset of $\RR^2$ can be written as an increasing union of subcontinua.

\begin{proof}
 If $F \subset \RR^2$ is closed and compactly connected, choose $x_0 \in F$ and let $K_n$ be the connected component of $x_0$ in $F \bigcap B(O,
n)$. We are left to show that $F \subset \bigcup_{n>0}K_n$ : but if $x\in F$, there exist a continuum $C$ such that $x_0, x \in C$ and therefore
$x \in K_n$ as soon as $C \subset B(O, n)$.
\end{proof}

\end{lem}

\begin{lem}
\label{boundary lem}
 Let $F$ be a closed, compactly connected, non compact, non separating subset of $\RR^2$, with int$F=\emptyset$, then any neighborhood of $F$
contains a neighborhood of $F$ homeomorphic to $\RR^2$ bounded by a proper line. Consequently, $\RR^2 \setminus F$ is homeomorphic to $\RR^2$.
\end{lem}

\begin{proof}
 Given any neighbohood $W$ of $F$, write $F$ as a union of compact, connected, non separating sets : $F = \bigcup _{n>0} K_n$, $K_n
\subset \mbox{\rm int}K_{n+1}$ and use Lemma~\ref{dlem} to find a ball $B_n$ such that $K_n \subset \mbox{\rm int}B_n \subset W$.
We choose $B_n$ as a subset a tiling of the plane by squares of side length decreasing with $n$. Then, the family (Fr$B_n)_{n>0}$ is locally finite
and if $V$ is the union of
$\bigcup _{n>0}B_n$ with all the components of $\RR^2 \setminus \bigcup _{n>0}B_n$ which lie inside $W$, then $\mbox{\rm Fr}V$ is a non
compact connected (since $F$ is non separating) one-manifold properly embedded in $\RR^2$, that is a proper line.
This implies that $\RR^2 \setminus F$ is
homeomorphic to an increasing sequence of half-planes and therefore homeomorphic to $\RR^2$.
\end{proof}

To prove Theorem~\ref{B}, according to Theorem~\ref{A}, we can suppose $F$ non compact and in
all the rest of this section, we consider $F$  a closed, compactly connected subset of $\RR^2$ without interior such that $\RR^2 \setminus F$
is homeomorphic to $\RR^2$ and a homeomorphism $h$ of $\RR^2$ preserving $F$: $h(F) = F$. 

It follows from Lemma~\ref{boundary lem} that $\infty$ is an accessible
point of $F \bigcup \{ \infty \}$ from ${\bf S}^2 \setminus F \bigcup \{ \infty \}$, and that if we let $\gamma_{\infty}$ be an access arc to $\infty$, 
then $h([\gamma_{\infty}])=[\gamma_{\infty}]$. Therefore the set of prime ends of ${\bf S}^2 \setminus F \bigcup \{ \infty \}$ minus $[\gamma_{\infty}]$,
 which we call the prime ends of $\RR^2 \setminus  F$, is linearly orderable, in fact homeomorphic to a line, and invariant under the homeomorphim induced by $h$.
Given a cyclic order on the prime ends of ${\bf S}^2 \setminus F \bigcup \{ \infty \}$, the prime ends of $\RR^2 \setminus  F$ are given by equivalence classes
of sequences $([p_n, p'_n])_{n\geq 0}$ with $p_n <p'_n <[\gamma_{\infty}]$ or  $[\gamma_{\infty}] <p_n <p'_n$ where $p_n$ and $p'_n$ can be represented by access arcs inside 
$\RR^2 \setminus F$ except for their endpoint in $F$.

To prove Theorem~\ref{B}, we will suppose, aiming to a contradiction, that $h$ has no fixed point in $F$ and therefore no fixed point on a
neighborhood $V$ of $F$.

\begin{lem}
\label{orderlem}
 On the line of prime ends $h$ has no fixed point and therefore, for any interval $[a, b]$, neither $[h(a), h(b)]$ or 
$[h^{-1}(a), h^{-1}(b)]$ is contained in $[a, b]$.
\end{lem}

\begin{proof} Suppose $q$, defined by $[[\gamma _n], [\gamma _n']]$, is a fixed prime end and let $c$ be a cut obtained by joining irreductibely
$\gamma_0$ to $\gamma'_0$ by an arc inside $\RR^2 \setminus F$. The
endpoints of 
$c$ are both in some compact connected subset $D$ of $F$ since $F$ is compactly connected and therefore,  the region cut out by $c$ in $\RR^2
\setminus
F$ and containing the endpoints of access arcs $\beta$ such that $[\beta] \in  [[\gamma _0], [\gamma _0']]$ is bounded. This implies that
the impression associated to $q = h(q)$ is a non separating compact connected set
invariant under $h$.
Theorem~\ref{A} would then give a fixed point of $h$ in this impression and
 therefore in $F$: a contradiction.
\end{proof}      
\medskip

Let $\gamma$ an access arc to some point $p_0 \in F$ short enough so that $h^{-1}(\gamma)$, $\gamma$, $h(\gamma)$, $h^2(\gamma)$ are all disjoint
(except perhaps in their endpoints)
and let $L$ be a proper line in $V \setminus F$, boundary of a neighborhood of $F$, close enough to $F$ so that $L$ meets $h^{-1}(\gamma)$, $\gamma$,
$h(\gamma)$, $h^2(\gamma)$.
There exists a subarc $pp_0$ of $\gamma$ joining irreductibely $L$ to $F$  
and the arc $pp_0$ separates the region $R$ between $L$ and $F$ into two sub-regions $A$ and $B$ which are
unbounded. Lemma~\ref{orderlem} says that $h(\gamma) \bigcap R$ and $h^{-1}(\gamma)\bigcap R$ are not in the same region and we call $A$ the one
containing
$h(\gamma) \bigcap R$ and $B$
the one containing $h^{-1}(\gamma) \bigcap R$. By definition $\gamma$ is on the frontiers of $A$ and $B$.  Notice that $A$ contains $h^k(\gamma)
\bigcap R,
k \geq 0$ and $B$ contains $h^{-k}(\gamma) \bigcap R, k \geq 1$. Also, all regions cut out in $R$ by $h^k(\gamma)$ and $h^{k+1}(\gamma)$, as $k$
varies in $\ZZ$, are disjoint (assuming that $h^k(\gamma)$ and $h^{k+1}(\gamma)$ meet $L$).

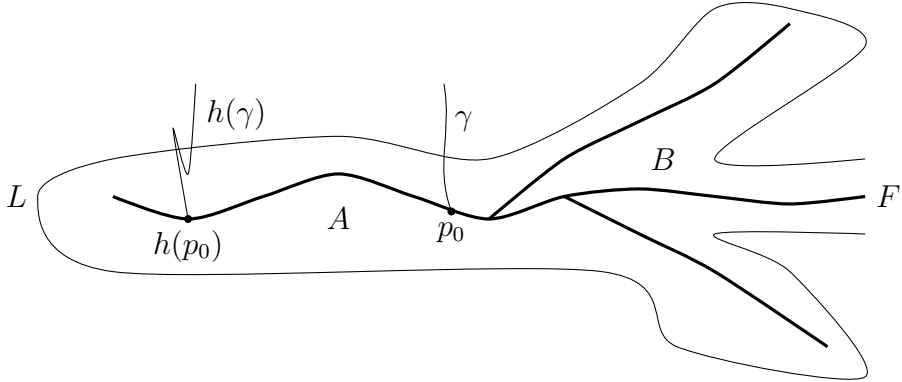
\begin{center}
\begin{tikzpicture}[scale=1] \useasboundingbox (0,0) rectangle (15,5); 
\draw [very thick] plot[smooth] coordinates{(2,2.5) (3,2.2) (4,2.5) (5,2.8) (6,2.5) (7,2.2) (8,2.5) (9,2.6) (10,2.5) (11,2.4) (12,2.5)}; 
\draw [very thick] plot[smooth] coordinates{(7,2.2) (8,3) (9,3.5) (10,4) (11,4.8)};
\draw [very thick] plot[smooth] coordinates{(8,2.5) (9,2) (10,1.5) (11.5,.5)};
\draw (12,2.5) node[right]{$F$};
\draw plot[smooth] coordinates{(12,3) (10,3) (12,4.5) (11,5) (10,5) (9,4) (7,3) (5,3.3) (2,3) (1,2.5) (2,1.5) (8.5,1.5) (9.5,.5) (12,.1) (11,1.5) (10,2) (12,2)};
\draw (5,2.5) node[below]{$A$};
\draw (9,3) node[right]{$B$};
\draw (1,2.5) node[left]{$L$};
\fill (3,2.2) circle (1.5pt);
\fill (6.5,2.3) circle (1.5pt);
\draw (3,2.2) node[below]{$h(p_0)$};
\draw (6.5,2.3) node[below]{$p_0$};
\draw [out=110,in=-80] (6.5,2.3) to (6.4,4);
\draw plot[smooth] coordinates{(3,2.2) (2.8,3.4) (3,2.8) (3.1,4)};
\draw (6.4,3.5) node[right]{$\gamma$};
\draw (3.1,3.6) node[right]{$h(\gamma)$};
\end{tikzpicture}
\end{center}

\begin{lem}
\label{compactlem}
 $\mbox{\rm Fr}A \bigcap \mbox{\rm Fr}B \bigcap F$ is unbounded.
\end{lem}

\begin{proof} Let $F_A$ (resp. $F_B$) be the set of points of $F$ which admit a neighborhood contained in $A \bigcup F$ (resp. $B \bigcup F$). The
sets 
$A \bigcup F_A$ and $B \bigcup F_B$ are disjoint (since $\mbox{\rm int}F = \emptyset$) and open, therefore their complement in $R \bigcup F
\bigcup \setminus (\gamma \setminus \{ p \})$ (which
complement is the set of points of $F $ for which every neighborhood meets $A$ and $B$, that is $\mbox{Fr}A \bigcap \mbox{Fr}B\bigcap F$)
separates $R
\bigcup F
\setminus (\gamma \setminus \{ p \})$ and $R \bigcup F \setminus (\gamma \setminus \{ p \})$ can be written as the disjoint union \hfill\break $(A
\bigcup
F_A) \coprod (B \bigcup F_B) \coprod ( \mbox{Fr}A \bigcap \mbox{Fr}B \bigcap F)$.

On the other hand, if $\mbox{Fr}A \bigcap \mbox{Fr}B \bigcap F$ was compact in $\RR^2$ or equivalently in $R \bigcup F$ (which is homeomorphic to
$\RR^2$), 
thinking of $L$ as a straight line and of $\gamma$ as a segment orthogonal to $L$ (as it is legitimate by Schoenflies theorem), one can find a
large
rectangle in $R \bigcup F$ with a side parallel to $L$ , containing $\mbox{Fr}A \bigcap \mbox{Fr}B \bigcap F$  and whose boundary cuts $\gamma$
transversaly in a
single point. The boundary of this rectangle joins a point of $A$ near $\gamma$ to a point of $B$ near $\gamma$ in contradiction to the above
decomposition
of  $R \bigcup F \setminus (\gamma \setminus \{ p \})$.
\end{proof}

Given the fact that $F$ is compactly connected, there exist a connected compact $K$  in $F$ containing $p_0$, $h(p_0)$
 and $h^{-1}(p_0)$. The preceding Lemma gives $x \in \mbox{Fr}A \bigcap \mbox{Fr}B \bigcap F$ outside
 of $K$ and we let $U$ be an euclidean disc inside $R \bigcup F$ containing $x$ such that $h(U) \bigcap U = \emptyset$. We can assume that $U$ does not meet
$K$, $\gamma$, $h(\gamma)$ and $h^{-1}(\gamma)$.

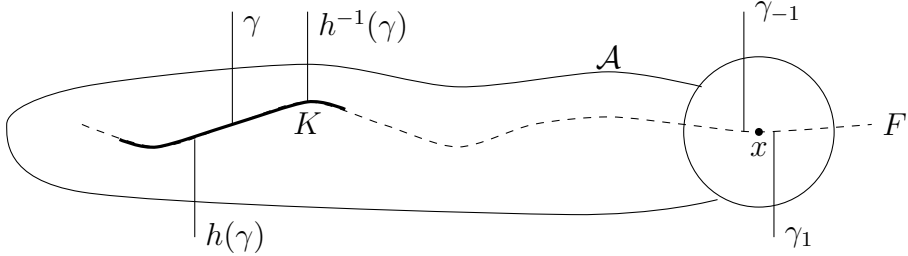
\begin{center}
\begin{tikzpicture}[scale=1] \useasboundingbox (0,0) rectangle (15,5); 
\draw [dashed] plot[smooth] coordinates{(2,2.5) (3,2.2) (4,2.5) (5,2.8) (6,2.5) (7,2.2) (8,2.5) (9,2.6) (10,2.5) (11,2.4) (12.5,2.5)}; 
\draw (12.5,2.5) node[right]{$F$};
\draw [very thick] plot[smooth] coordinates{(2.5,2.3) (3,2.2) (4,2.5) (5,2.8) (5.5,2.7)};
\draw (5,2.8) node[below]{$K$};
\fill (11,2.4) circle (1.5pt);
\draw (11,2.4) node[below]{$x$};
\draw plot[smooth] coordinates{(10.24,3) (9,3.2) (7,3) (5,3.3) (2,3) (1,2.5) (2,1.6) (8.5,1.3) (10.45,1.5)};
\draw (11,2.4) circle (1);
\draw (4,2.5)--(4,4);
\draw (4,3.8) node[right]{$\gamma$};
\draw (5,2.8)--(5,4);
\draw (5,3.8) node[right]{$h^{-1}(\gamma)$};
\draw (3.5,2.3)--(3.5,1);
\draw (3.5,1) node[right]{$h(\gamma)$};
\draw (10.8,2.4)--(10.8,4);
\draw (11.2,2.4)--(11.2,1);
\draw (10.8,4) node[right]{$\gamma_{-1}$};
\draw (11.2,1) node[right]{$\gamma_1$};
\draw (9,3.1) node[above]{${\cal A}$};
\end{tikzpicture}
\end{center}

Choose any arc joining the boundary of $U$ to itself in the region $R$ between $L$ and $F$ so that, united to an arc of $\partial U$ it gives a
simple 
closed curve $\cal A$ containing $K$ and $x$ (to get such an arc, we can consider the boundary of a neighborhood of a continuum in $F$ containing $K \bigcup \{ x\}$). 
Choose then $\epsilon >0$ so that dist$(h(u), v)>\epsilon$ for all $u, v$ inside the
curve ${\cal A}$ with dist$(u, v)< 3\epsilon$. In particular, if $\mbox{\rm diam}X <3\epsilon$, then $h(X) \bigcap X = \emptyset$. We also ask
that $\epsilon< \mbox{\rm dist}(\gamma \bigcap R, h(\gamma)\bigcap R)$. Choose also
$\delta$, $\epsilon >\delta >0$, such that $\mbox{dist}(u, v)<\delta$ implies $\mbox{dist}(h(u), h(v))<\epsilon$ and $\mbox{dist}(h^{-1}(u),
h^{-1}(v))<\epsilon$ for $u, v \in \mbox{int}{\cal A}$.

Let $p_ {-1}$ a point of $U \bigcap F$ accessible from $B$, $p_1$ a point of $U \bigcap F$ accessible from $A$ with corresponding access arcs 
$\gamma _{-1}$ and $\gamma _1$ such that $\mbox{\rm dist}(p_{-1}, p_1)<\epsilon$. We order the line of prime ends so that $[\gamma_{-1}]<[\gamma]<[\gamma_1]$. 
By Lemma~\ref{orderlem}, exchanging $h$ and $h^{-1}$ if necessary, we can (and we will) suppose that $h(p)>p>h^{-1}(p)$ for every prime end $p$.
We can assume that $\gamma_{-1}$, $h(\gamma_{-1})$,  $\gamma$, 
 $h^{-1}(\gamma_1)$, $\gamma_1$, $h(\gamma_1)$ are all disjoint and meet $L$ (choosing, if necessary, a new $L$ closer to $F$).
By construction $\gamma $ separates $h(\gamma _{-1})$ and $h^{-1}(\gamma _1)$ (inside $R$) and we
have the order $[\gamma_{-1}]<[h(\gamma_{-1})]<[\gamma]<[h^{-1}(\gamma_1)]<[\gamma_1]$.

\subsection{A simple closed curve}

We consider the rectilinear segment $p_1p_{-1}$ and $u_{-1}$ the infimum on the line of prime ends of classes of access arcs between $[\gamma]$
and $[\gamma_{-1}]$ which do not cut $p_1p_{-1}$ and $u_1$ the supremum of classes of access arcs between $[\gamma]$
and $[\gamma_1]$ which do not cut $p_1p_{-1}$. Choose a disc $B$ in the $\epsilon$-neighborhood of $F$ as given by Lemma~\ref{dbislem} applied to
a subcontinuum $M$ of $F$ containing $F \bigcap \overline{\mbox{\rm int}{\cal A}}$, $p_1$ and $p_{-1}$, and to the arcs $\gamma_1$, $\gamma_{-1}$.
We choose $B$ close enough to $M$ so that Fr$B$ meets $\gamma_1$ and $\gamma_{-1}$.
Let $c$ be a cut inside $\cal A$ made of three arcs : two end parts $\rho_-$, $\rho_+$ of access arcs such that $[\rho_-]<u_{-1}<[\rho_+]$,
$\rho_- \bigcap \gamma_{-1} = \emptyset =\rho_+ \bigcap \gamma_{-1}$ and an arc of Fr$B$. We can suppose that $c$ and $h(c)$ do not meet
$\gamma$ and that diam$c<3\epsilon$ so that $c \bigcap h(c)=\emptyset$. Similarly, we define a cut $d$ inside ${\cal A}$ made of (end parts) of
access arcs $\mu_-$ and $\mu_+$, $[\mu_-]<u_1<[\mu_+]$ and an arc of Fr$B$ such that $d \bigcap \gamma = \emptyset =h^{-1}(d) \bigcap \gamma$ and
$d\bigcap h^{-1}(d)=\emptyset$.

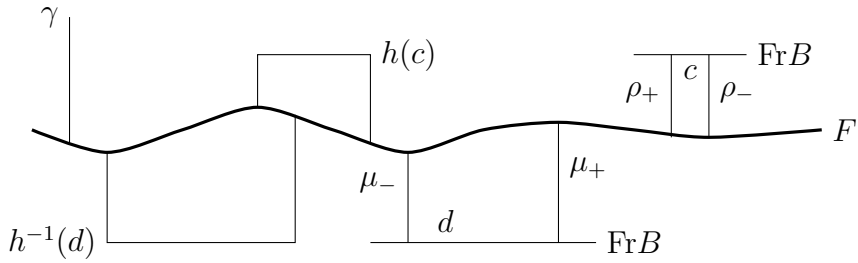
\begin{center}
\begin{tikzpicture}[scale=1] \useasboundingbox (0,0) rectangle (15,5); 
\draw [very thick] plot[smooth] coordinates{(2,2.5) (3,2.2) (4,2.5) (5,2.8) (6,2.5) (7,2.2) (8,2.5) (9,2.6) (10,2.5) (11,2.4) (12.5,2.5)}; 
\draw (12.5,2.5) node[right]{$F$};
\draw (2.5,2.3)--(2.5,4);
\draw (2.5,4) node[left]{$\gamma$};
\draw (11,2.4)--(11,3.5);
\draw (11,3) node[right]{$\rho_-$};
\draw (10.5,2.4)--(10.5,3.5);
\draw (10.5,3) node[left]{$\rho_+$};
\draw (10,3.5)--(11.5,3.5);
\draw (11.5,3.5) node[right]{Fr$B$};
\draw (10.75,3.5) node[below]{$c$};
\draw (9,2.6)--(9,1);
\draw (9,2) node[right]{$\mu_+$};
\draw (7,2.2)--(7,1);
\draw (7,1.8) node[left]{$\mu_-$};
\draw (6.5,1)--(9.5,1);
\draw (9.5,1) node[right]{Fr$B$};
\draw (7.5,1) node[above]{$d$};
\draw (5,2.8)--(5,3.5)--(6.5,3.5)--(6.5,2.3);
\draw (6.5,3.5) node[right]{$h(c)$};
\draw (5.5,2.7)--(5.5,1)--(3,1)--(3,2.2);
\draw (3,1) node[left]{$h^{-1}(d)$};

\end{tikzpicture}
\end{center}

Now, choose a sequence of discs $(B_n)_{n\geq 1}$ inside $B$ as given by Lemma~\ref{dbislem} relative to $M$ and the arcs  $\rho_+$,
$h(\rho_-)$, $h(\rho_+)$, $h^{-1}(\mu_-)$. 

If $u_1 >[\gamma_1]$, that is, if between $[\rho_-]$ and $[\rho_+]$ on the line of prime ends there exist
 prime ends of the form $[\delta]$ such that $\delta$ cuts the segment $p_1p_{-1}$ for every representative $\delta$, then we choose $n_1$ large
enough so that Fr$B_{n_1}$ cuts $p_1p_{-1}$ inside the bounded region determined by $c$ and $M$ and we call $r$ the last point of intersection on
$p_1p_{-1}$ of $p_1p_{-1}$ with Fr$B_{n_1}$.

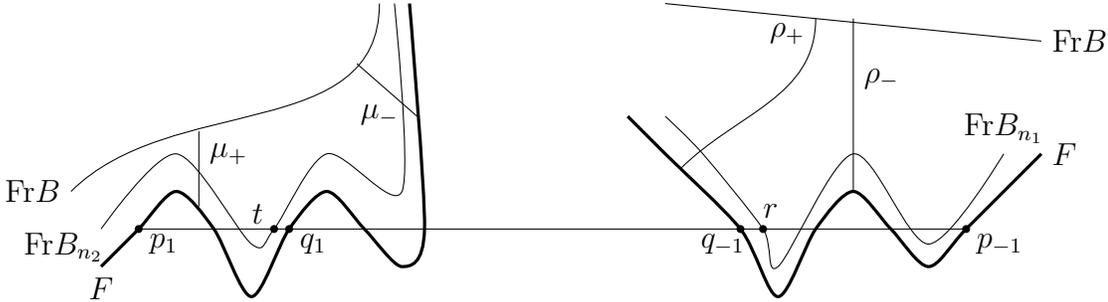
\begin{center}
\begin{tikzpicture}[scale=1] \useasboundingbox (0,0) rectangle (15,5); 
\draw (1,1)--(12,1);
\draw (12,.8) node[right]{$p_{-1}$};
\fill (12,1) circle (1.5pt);
\draw [very thick] plot[smooth] coordinates {(13,2) (12,1) (11.5,.5) (11,1) (10.5,1.5) (10,1) (9.5,.1) (9,1) (8,2) (7.5,2.5)};
\draw (13,2) node[right]{$F$};
\draw plot[smooth] coordinates {(12.5,2) (11.5,.8) (10.5,2) (9.5,.5) (9.3,1)(8.5,2) (8,2.5)};
\draw (12.5,2) node[above]{Fr$B_{n_1}$};
\fill (9,1) circle (1.5pt);
\draw (9.2,.8) node[left]{$q_{-1}$};
\fill (9.3,1) circle (1.5pt);
\draw (9.4,1) node[above]{$r$};
\draw (13,3.5)--(8,4);
\draw (13,3.5) node[right]{Fr$B$};
\draw (10.5,1.5)--(10.5,3.8);
\draw (10.5,3) node[right]{$\rho_-$};
\draw [out=45,in=-90] (8.2,1.8) to (10,3.8);
\draw (10,3.6) node[left]{$\rho_+$};
\draw [very thick] plot[smooth] coordinates {(.5,.5) (1,1) (1.5,1.5) (2,1) (2.5,.1) (3,1) (3.5,1.5) (4,1) (4.5,.5) (4.8,1) (4.6,4)};
\fill (1,1) circle (1.5pt);
\fill (3,1) circle (1.5pt);
\draw (.5,.5) node[below]{$F$};
\draw (1,.8) node[right]{$p_1$};
\draw (3,.8) node[right]{$q_1$};
\draw plot[smooth] coordinates {(.5,1) (1.5,2) (2.5,.8) (2.8,1) (3.5,2) (4.5,1.5) (4.4,4)};
\fill (2.8,1) circle (1.5pt);
\draw (2.8, 1.2) node[left]{$t$};
\draw (0,1.1) node[below]{Fr$B_{n_2}$};
\draw [out=45,in=-90] (.1,1.5) to (4.2,4);
\draw (.1,1.5) node[left]{Fr$B$};
\draw (1.8,1.3)--(1.8,2.3);
\draw (1.8,2) node[right]{$\mu_+$};
\draw (4.7,2.5)--(3.9,3.2);
\draw (4.2,2.8) node[below]{$\mu_-$};
\end{tikzpicture}
\end{center}

In the opposite case (that is $u_{-1} = [\gamma_{-1}]$) we let $r$ be the last point on $\gamma_{-1}$ (from $x_0$ to $F$) of $\gamma_{-1} \bigcap
\mbox{\rm Fr}B_{n_1}$ where $n_1$ is choosen large enough so that the diameter of the subarc $rp_{-1}$ of $\gamma_{-1}$ is less than $\epsilon$.

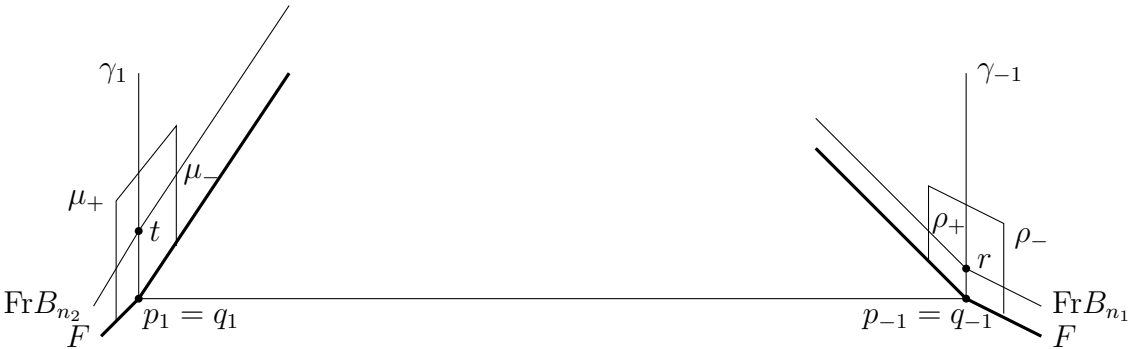
\begin{center}
\begin{tikzpicture}[scale=1] \useasboundingbox (0,0) rectangle (15,5); 
\draw (1,1)--(12,1);
\draw [very thick] (13,.5)--(12,1)--(10,3);
\draw (13,.9)--(12,1.4)--(10,3.4);
\draw (12,1)--(12,4);
\draw (12.5,.8)--(12.5,2)--(11.5,2.5)--(11.5,1.5);
\fill (12,1) circle (1.5pt);
\fill (12,1.4) circle (1.5pt);
\draw (13,.5) node[right]{$F$};
\draw (13,.9) node[right]{Fr$B_{n_1}$};
\draw (12,4) node[right]{$\gamma_{-1}$};
\draw (12.5,1.8) node[right]{$\rho_-$};
\draw (11.4,2) node[right]{$\rho_+$};
\draw (12,1.5) node[right]{$r$};
\draw (11.5,1) node[below]{$p_{-1}=q_{-1}$};
\draw [very thick] (.5,.5)--(1,1)--(3,4);
\draw (.5,.5) node[left]{$F$};
\draw (.4,.9)--(1,1.9)--(3,4.9);
\draw (.4,.9) node[left]{Fr$B_{n_2}$};
\fill (1,1) circle (1.5pt);
\fill (1,1.9) circle (1.5pt);
\draw (1,1)--(1,4);
\draw ((1,4) node[left]{$\gamma_1$};
\draw (.7,.7)--(.7,2.3)--(1.5,3.3)--(1.5,1.7);
\draw (1,1.9) node[right]{$t$};
\draw (1.7,1) node[below]{$p_1=q_1$};
\draw (.7,2.3) node[left]{$\mu_+$};
\draw ((1.45,2.65) node[right]{$\mu_-$};
\end{tikzpicture}
\end{center}

Similarly, if between $[\mu_-]$ and $[\mu_+]$ on the line of prime ends there exist
 prime ends of the form $[\delta]$ such that $\delta$ cuts the segment $p_1p_{-1}$ for every representative $\delta$, then we choose $n_2$ large
enough so that Fr$B_{n_2}$ cuts $p_1p_{-1}$ inside the bounded region determined by $d$ and $M$ and we call $t$ the first point of intersection on
$p_1p_{-1}$ of $p_1p_{-1}$ with Fr$B_{n_2}$.

In the opposite case (that is $u_1 = [\gamma_1]$) we let $t$ be the last point on $\gamma_1$ (from $x_0$ to $F$) of $\gamma_1 \bigcap
\mbox{\rm Fr}B_{n_2}$ where $n_2$ is choosen large enough so that the diameter of the subarc $tp_1$ of $\gamma_1$ is less than $\epsilon$.

On the arc $rt$ (which is a subarc of $\gamma_{-1} p_{-1}p_1\gamma_1$), let $q_{-1}$ and $q_1$ be the first and last point on $M$. We now
consider the cut $c'$ made of the arc $q_{-1}r$, the subarc of Fr$B_{n_1}$ from $r$ to $\rho_+$ and the subarc of $\rho_+$ from this last point
to $F$, and we join irreductibly  the part of $\rho_+$ in $c'$ to $h(q_{-1}r)$ by an arc of Fr$B_{m_1}$ for some $m_1>n_1$ large enough so that
Fr$B_{m_1}$ cuts the part of $\rho_+$ in $c'$ and $h(q_{-1}r)$.

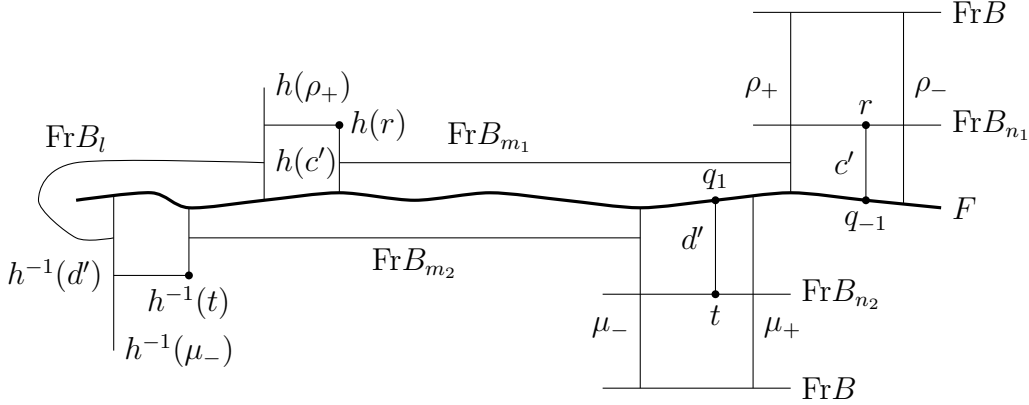
\begin{center}
\begin{tikzpicture}[scale=1] \useasboundingbox (0,0) rectangle (15,5); 

\draw [very thick] plot[smooth] coordinates{(.5,2.5) (1.5,2.6) (2,2.4) (3,2.5) (4,2.6) (5,2.5) (6,2.6) (7,2.5) (8,2.4) (9,2.5) (10,2.6) (11,2.5) (12,2.4)};
\draw (12,2.4) node[right]{$F$};
\draw (11.5,2.45)--(11.5,5);
\draw (10,5)--(10,2.6);
\draw (9.5,5)--(12,5);
\draw (12,5) node[right]{Fr$B$};
\draw (9.5,3.5)--(12,3.5);
\draw (12,3.5) node[right]{Fr$B_{n_1}$};
\draw (11,2.5)--(11,3.5);
\fill (11,2.5) circle (1.5pt);
\fill (11,3.5) circle (1.5pt);
\draw (11,3.5) node[above]{$r$};
\draw (11,2.5) node[below]{$q_{-1}$};
\draw (11.5,4) node[right]{$\rho_-$};
\draw (10,4) node[left]{$\rho_+$};
\draw (11,3) node[left]{$c'$};
\draw (9.5,2.55)--(9.5,0);
\draw (8,2.4)--(8,0);
\draw (7.5,0)--(10,0);
\draw (7.5,1.25)--(10,1.25);
\draw (10,0) node[right]{Fr$B$};
\draw (10,1.25) node[right]{Fr$B_{n_2}$};
\draw (9.5,.8) node[right]{$\mu_+$};
\draw (8,.8) node[left]{$\mu_-$};
\draw (9,2.5)--(9,1.25);
\fill (9,2.5) circle (1.5pt);
\fill (9,1.25) circle (1.5pt);
\draw (9,1.25) node[below]{$t$};
\draw (9,2.5) node[above]{$q_1$};
\draw (9,2) node[left]{$d'$};
\draw (1,2.55)--(1,.5);
\draw (1,.5) node[right]{$h^{-1}(\mu_-)$};
\draw (2,2.4)--(2,1.5);
\draw (1,1.5)--(2,1.5);
\fill (2,1.5) circle (1.5pt);
\draw (2,1.5) node[below]{$h^{-1}(t)$};
\draw (1,1.5) node[left]{$h^{-1}(d')$};
\draw (3,2.5)--(3,4);
\draw (3,4) node[right]{$h(\rho_+)$};
\draw (4,2.6)--(4,3.5);
\fill (4,3.5) circle (1.5pt);
\draw (4,3.5) node[right]{$h(r)$};
\draw (4,3.5)--(3,3.5);
\draw (3,3) node[right]{$h(c')$};
\draw (4,3)--(10,3);
\draw (6,3) node[above]{Fr$B_{m_1}$};
\draw plot[smooth] coordinates{(3,3) (.5,3) (0,2.5) (.5,2) (1,2)};
\draw (.5,3) node[above]{Fr$B_l$};
\draw (2,2)--(8,2);
\draw (5,2) node[below]{Fr$B_{m_2}$};
\end{tikzpicture}
\end{center}

Similarly,  we now
consider the cut $d'$ made of the arc $q_1 t$, the subarc of Fr$B_{n_2}$ from $t$ to $\mu_-$ and the subarc of $\mu_-$ from this last point
to $F$, and we join irreductibly  the part of $\mu_-$ in $d'$ to $h^{-1}(q_1 t)$ by an arc of Fr$B_{m_2}$ for some $m_2>n_2$ large enough so
that
Fr$B_{m_2}$ cuts the part of $\mu_-$ in $c'$ and $h^{-1}(q_1 t)$.

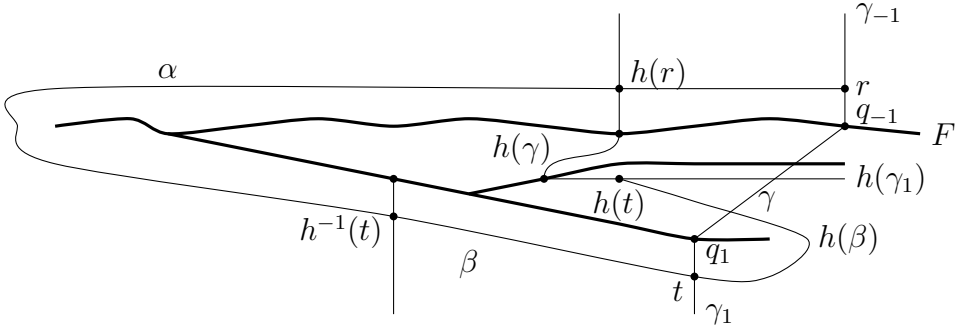
\begin{center}
\begin{tikzpicture}[scale=1] \useasboundingbox (0,0) rectangle (15,5); 

\draw [very thick] plot[smooth] coordinates{(.5,2.5) (1.5,2.6) (2,2.4) (3,2.5) (4,2.6) (5,2.5) (6,2.6) (7,2.5) (8,2.4) (9,2.5) (10,2.6) (11,2.5) (12,2.4)};
\draw [very thick] plot[smooth] coordinates{(2,2.4) (3,2.2) (4,2) (5,1.8) (6,1.6) (7,1.4) (8,1.2) (9,1) (10,1)};
\draw [very thick] plot[smooth] coordinates{(6,1.6) (7,1.8) (8,2) (9,2) (10,2) (11,2)};
\draw (12,2.4) node[right]{$F$};
\draw (11,2.5)--(11,4);
\fill (11,2.5) circle (1.5pt);
\fill (11,3) circle (1.5pt);
\draw (11,3) node[right]{$r$};
\draw (11,2.65) node[right]{$q_{-1}$};
\draw (11,4) node[right]{$\gamma_{-1}$};
\draw (8,2.4)--(8,4);
\fill (8,2.4) circle (1.5pt);
\fill (8,3) circle (1.5pt);
\draw (8,3.2) node[right]{$h(r)$};
\draw (9,1)--(11,2.5);
\draw (9,1)--(9,0);
\fill (9,1) circle (1.5pt);
\fill (9,.5) circle (1.5pt);
\draw (9,.8) node[right]{$q_1$};
\draw (9,.3) node[left]{$t$};
\draw (9,0) node[right]{$\gamma_1$};
\draw (5,1.8)--(5,0);
\fill (5,1.3) circle (1.5pt);
\fill (5,1.8) circle (1.5pt);
\draw (5,1.1) node[left]{$h^{-1}(t)$};
\draw (7,1.8)--(11,1.8);
\draw (11,1.8) node[right]{$h(\gamma_1)$};
\fill (7,1.8) circle (1.5pt);
\draw plot[smooth] coordinates{(11,3) (8,3) (.5,3) (0,2.5) (.5,2) (5,1.3) (9,.5) (10,.5) (10.5,1) (9,1.5) (8,1.8)};
\fill (8,1.8) circle (1.5pt);
\draw (8,1.8) node[below]{$h(t)$};
\draw (10.5,1) node[right]{$h(\beta)$};
\draw (2,3) node[above]{$\alpha$};
\draw (6,1) node[below]{$\beta$};
\draw (9.7,1.5) node[right]{$\gamma$};
\draw [out=90,in=-90] (7,1.8) to (8,2.4);
\draw (7.25,2.2) node[left]{$h(\gamma)$};

\end{tikzpicture}
\end{center}

Finally, we join irreductibly $h(\rho_+)$ and $h^{-1}(\mu_-)$ by a subarc of Fr$B_l$, $l>m_1, m_2$, and we consider the simple closed curve from
$r$ to $h(r)$ to $h^{-1}(t)$, to $t$ to $r$ composed of subarcs of the cuts $c'$, $h(c')$, $h^{-1}(d')$, $d'$, of subarcs of Fr$B_{m_1}$,
Fr$B_l$, Fr$B_{m_2}$ and of the arc $tr$.

In every case, we have constructed a simple closed curve $\cal C$ in the $\epsilon$-neighborhood of $F$ composed of three consecutive arcs
$\alpha$ (from $r$ to $h^{-1}(t)$), $\beta$ (from $h^{-1}(t)$ to $t$), $\gamma$ (from $t$ to $r$) with disjoints
interiors such that, if $O$ denotes the origin of $\alpha$, $h(O) \in \alpha$, the endpoint of $\beta$ is the origin of $h(\beta)$ and $h(\beta)
\bigcap \beta$ is reduced to the endpoint of $\beta$, $h(\beta) \bigcap \alpha = \emptyset$ and $h(\gamma ) \bigcap \gamma = \emptyset$.
Furthermore, according to Lemma~\ref{dbislem}, we can find a sequence of points $r_0 = r, r_1, r_2, \ldots, r_n = h^{-1}(t)$ on $\alpha$ and for
each $i, 1 \leq i \leq n-1$ an arc $\rho_i$ inside
$\cal C$, in the region determined by $\gamma_{-1}$, $h^{-1}(\gamma_1)$ and $F$,
irreducible from $r_i$ to $F$ such that $\rho_i \bigcap \rho_{i+1}$ is at most a point in $F$. We can make these choices 
(choosing $n_1$, $m_1$, $l$ large enough) so that diam$\rho_ir_ir_{i+1}\rho_{i+1}$ is less than $3\epsilon$ so that
$\rho_ir_ir_{i+1}\rho_{i+1}$ is disjoint from its image under $h$.

\subsection{An index computation}

Our aim is now to compute the index of the non vanishing vector field $\zeta (u) = h(u)-u$ along $\cal C$: if it is non zero, we will have reached
the desired final contradiction.

We observe that an arc $\rho$ from $\alpha \subset {\cal C}$ towards
$F$ in $\overline {\mbox{\rm int}{\cal C}}$ preceding $h^{-1}(\gamma_1)$ (on the line of prime ends) verifies that $h(\rho)$, from $h(\alpha)$ to
$F$, precedes $\gamma_1$ and so must meet $\mbox{int}{\cal C}$.
Therefore, hypothesis 3) of Lemma~\ref{Kbislem} is satisfied. Lemma~\ref{orderlem} justifies hypothesis 4) and also says that in fact
 $h(\rho_ir_ir_{i+1}\rho_{i+1})$ lies outside the bounded region
$\Omega _i$ cut out from ${\bf R}^2 \setminus F$ by $\rho_ir_ir_{i+1}\rho_{i+1}$. Therefore, we can find a new non vanishing vector field 
 which points inward the curve $\cal C$ or on $\cal C$ as its origin describes the arc $\alpha$ and
which has the same index as $\zeta$ on $\cal C$.

Since $h(\gamma) \bigcap \gamma = \emptyset$, the distance between $h(\gamma)$ and the endpoint of $\beta$ (which is also the origin of $\gamma$)
is positive and we can find an isotopy supported in a small neighborhood, disjoint of $\gamma$, of a closed subarc of $\beta \setminus \mbox{\rm endpoint}(\beta)$ which moves
$h(\gamma)$ outside of $\beta$. This gives a new non zero vector field, which we write as $f(x)-x$, on $\cal C$ with the same index as $\zeta$.
Since the origin of $\beta$ can be joined to $\infty$ using $h^{-1}(\gamma_1)$, we see that $\beta$ lies in the unbounded region of $\RR^2
\setminus f(\cal C)$ except for its endpoint and we can apply Lemma
\ref{Tlem} to conclude that our original vector field has index 1, which concludes the proof by contradiction of Theorem~\ref{B}.

\begin{exa}
 Simple examples show the necessity of the hypothesis. For $\mbox{\rm int}F = \emptyset$, consider a translation and an invariant
half-plane, for $\RR^2 \setminus F $ connected, consider a translation and an invariant line. For $F$ compactly connected, consider the
translation $\tau$ given by $\tau (x, y) = (x+2, y)$ and for $F$ the set $\bigcup _{n \in \ZZ} \tau ^n (G)$ where $G$ is the union of the
half-lines $\{ (0, y), y \geq 0 \}$, $\{ (1, y), y \geq 0 \}$,  $\{ (2, y), y \geq 0 \}$ and of all the segments from $(1, n)$ to $(\dfrac{1}{n},
0)$ and to $(2-\dfrac{1}{n},0)$, $n\geq 2$.
\end{exa}

\begin{rem}
In \cite{Guillou2011}, Theorem~\ref{B} is proved under the further assumption that $h$ is fixed point free on ${\RR}^2 \setminus F$ using some
Brouwer theory
related to the plane translation theorem \cite{Brouwer1912} (compare to Remark 3.2). One can reduce the present Theorem~\ref{B} to the one in
\cite{Guillou2011} using
a covering argument as in \cite{Brown}.
 
\end{rem}

\section{Proof of Theorem~\ref{C}}

We will need some more elementary index computations.

\begin{lem}
\label{ilem5}
Let $\cal C$ be a simple closed curve positively oriented (int${\cal C}$ is on the left) in the plane and $\alpha \subset {\cal C}$ an arc from
$a$
to $b$. Let also $f, g : \alpha \rightarrow \RR^2$ be maps without fixed point such that $f(\alpha) \subset \mbox{int}{\cal C} \bigcup {\cal C}$ and 
$g(\alpha) \subset \mbox{ext}{\cal C} \bigcup {\cal C}$.
\begin{enumerate}
 \item If $f(a) = g(a)$ and $f(b)=g(b)$ lie in int$\alpha$, then $i(f, \alpha)-i(g, \alpha)=-1$. Idem if $a$ and $b$ lie inside the arc of 
 ${\cal C}$ from $f(a) = g(a)$ to $f(b)=g(b)$.
 
\item If $f(a)=g(a)$ lies in int$\alpha$ and $fb)=g(b)$ lies after $b$ outside of $\alpha$ on $\cal C$, then $i(f, \alpha)-i(g, \alpha)=0$.
Idem if $f(b)=g(b)$ lies in int$\alpha$ and $f(a)=g(a)$ lies before $a$ outside of $\alpha$.
\end{enumerate}
\end{lem}
\begin{proof}
As for 1), using Schoenflies theorem, we can think of the arc $ab$  as a vertical segment with $\cal C$ on the left of the line
supporting this
segment. In the first case, the vector $\dfrac{f(\alpha(t)) -\alpha(t)}{\vert f(\alpha(t))-\alpha(t) \vert}$ goes from $(0, 1)$ to $(0, -1)$
without ever pointing to the right so that $i(f, \alpha )= -\dfrac{1}{2}$ and similarly $i(g, \alpha)=\dfrac{1}{2}$. The second case and point 2)
are treated in the same way.
 \end{proof}

We consider now a non degenerated non separating compact connected set $K \subset {\bf R}^2$  and
an orientation preserving homeomorphism $h : {\bf R}^2 \rightarrow {\bf R}^2$ preserving $K$: $h(K) = K$. According to Theorem~\ref{A}, $h$ admits
a fixed point in $K$.

The simple case $K = [-1, 1] \times \{0\} \subset {\bf R}^2$ and $h$ a $\pi$-rotation around $(0, 0)$, shows that some extra
hypothesis in Theorem~\ref{A} is needed in order to get two fixed points in $K$. What follows is a formal version of the idea of preserving the
sides of $[0, 1] \times \{0\}$.

We suppose further that the circle of prime ends of ${\bf R}^2 \setminus K$ splits into two (non degenerated) arcs $a_1$ and $a_2$ with the same
endpoints such that
$\bigcup _{p \in a_i}{\rm I}(p) = K$, $i = 1, 2$, where ${\rm I}(p)$ is the impression of the prime end $p$ (and therefore int$K = \emptyset$).

Theorem~\ref{C} then states that if the orientation preserving homeomorphism of the circle of prime ends induced by $h$ preserves $a_1$
and $a_2$ (that is fixes the common endpoints of $a_1$ and $a_2$), then $h$ admits two fixed points in $K$

For the proof we will argue by contradiction and suppose that $h$ has only one fixed point $p_0 \in K$. 

\begin{lem}
 There is no accessible periodic point of period $k>1$ in $K$.
\begin{proof}
 Suppose there exist an accessible periodic point $p$ in $K$ of period $k>1$ and let $\gamma$ be an access arc for $p$. We can 
 suppose that $[\gamma] \in a_1$ and, since the orientation preserving homeomorphim of the circle of prime ends induced by $h$ has no periodic point (only fixed points), 
 we can suppose that $[\gamma]$, $[h(\gamma])$, ..., $[h^k(\gamma)]$ are represented by disjoint arcs $\gamma_0$, $\gamma_1$, ..., $\gamma_k$
 except that $\gamma_0 \bigcap \gamma_k = \{ p\}$. We can find Jordan curve $J$ by joining irreductibely $\gamma$ and $\gamma_k$ inside $\RR^2 \setminus K$,
 such that int$J$ contains $h(p)$, ..., $h^{k-1}(p)$,
 so that int$J \bigcap K \not = \emptyset$ but no point of int$J \bigcap K$ is endpoint of an access arc $\delta$ with $[\delta] \in a_2$: a contradiction.
\end{proof}
\end{lem}

\begin{lem}
 There exist access arcs $\gamma$ with $[\gamma ] \in a_1$ and $\delta$ with $[\delta] \in a_2$ with endpoints $p$ and $q$ and an
arc $\alpha$ from $p$ to $q$ such that $\gamma \bigcup \alpha \bigcup \delta$ is a arc in $\RR^2$, $h(\alpha) \bigcap \alpha =
\emptyset = h^2(\alpha) \bigcap \alpha$ and $h(\gamma) \bigcap \alpha =\emptyset = h^{-1}(\gamma) \bigcap \alpha$.
 (Possibly $p=q$ and $\alpha$ is reduced to a point).
\begin{proof}
 Choose some access arc $\gamma$ with
$[\gamma] \in a_1$ and with endpoint $p \not = p_0$ in $K$. Using Schoenflies theorem one
can think of $\gamma$ as a straight segment. Let then $B$ be an euclidean disc such that $p
\in \mbox{int}B$, $p_0 \notin B$, $B \bigcap h(\gamma) = \emptyset = B \bigcap h^{-1}(\gamma)$, $B \bigcap h^{-1}(B) = \emptyset = B \bigcap h(B)$
and $B \bigcap h^2(B)=\emptyset$. 
 Let $\tilde \delta$ be an access arc with $[\tilde \delta] \in a_2$ and with endpoint $\tilde q \in B \bigcap K$. Inside $B$ the segment $\tilde \alpha$ from
$p$ to $\tilde q$ satisfies $\tilde \alpha \bigcap \gamma = \{ p \}$. 
 We can suppose $\tilde \delta$ short enough so that $\tilde \delta \subset B$ and therefore $\tilde \delta \bigcap h(\tilde \alpha) = \emptyset$. 
 We now follow $\tilde \delta$ from its origin to $\tilde q$ until we meet
$\tilde \alpha$. We then follow $\tilde \alpha$
towards $p$ until we reach an
accessible point $q$ on $K \bigcap \tilde \alpha$; the path followed is an access arc $\delta$ for $q$ and we define $\alpha$ as the part of
$\tilde \alpha$ between $p$ and $q$. Surely $[\delta] \in a_2$ for otherwise, between $[\delta]$ and $[\tilde \delta]$ there would be an endpoint of $a_1$
which is a fixed prime end in contradiction to $B \bigcap h(B) = \emptyset$.
\end{proof}

\end{lem}

For the rest of this section we suppose, without loss of generality, that $g=h$. We will also assume that $\gamma$ and $\delta$ are short enough so that the eight arcs 
$h^{-1}(\gamma)$, $\gamma$, $h(\gamma)$, $h^2(\gamma)$ and $h^{-1}(\delta)$, $\delta$, $h(\delta)$, $h^2(\delta)$ are all disjoints.

\medskip

Let $\cal A$ be a simple closed curve such that $K^+=K \bigcup \alpha \bigcup h(\alpha) \bigcup h^2(\alpha) \subset \mbox{\rm int}{\cal A}$ close enough
to $K^+$ so that $\cal A$ cuts $h^{-1}(\gamma)$,
 $\gamma$, $h(\gamma)$, $h^2(\gamma)$ and $h^{-1}(\delta)$, $\delta$, $h(\delta)$, $h^2(\delta)$ . That
curve is split by an irreductible subarc of $\gamma \bigcup \alpha \bigcup \delta$ from $\cal A$ to itself, containing
$\alpha$, into two arcs ${\cal A}_1$ and ${\cal A}_2$ with the same endpoints. These arcs,
joined with the preceding irreducible subarc of $\gamma \bigcup \alpha \bigcup \delta$, give rise to two simple
closed curves $\tilde {\cal A}_1$ and $\tilde {\cal A}_2$ with disjoint interiors such that one of them, say $\tilde {\cal A}_1$, does not contain
the fixed point $p_0$. Denote by $\tilde{\cal D}_1$ the closure of the interior of $\tilde{\cal A}_1$ and let $L = K \bigcap (\tilde {\cal D}_1
\bigcup h(\tilde {\cal D}_1) \bigcup h^2(\tilde {\cal D}_1))$. Notice that $p_0 \notin L$.

We orient $\cal A$  by going from  $\gamma$ to $\delta$ on $\tilde {\cal
A}_1$ without meeting $\alpha$ (equivalently, if $\alpha$ is non degenerated, we orient $\alpha$ from $q$ to $p$).

Given our hypothesis that there is only $p_0$ as fixed point in $K$, we can now find a neighborhood $U$ of $L \bigcup \alpha \bigcup h(\alpha)
\bigcup h^2(\alpha)$ such that $p_0 \notin U$ and $\epsilon >0$ such that
dist$(h(x), x) > 3\epsilon$ and dist$(h^{-1}(x), x) >3\epsilon$ on $U$.
Furthermore, we ask that $2\epsilon < \mbox{\rm dist}(\gamma \bigcap U, h(\gamma) \bigcap U), \mbox{\rm dist}(\delta \bigcap U, h(\delta) \bigcap
U)$, 
$2\epsilon
< \mbox{dist}(\alpha, h(\alpha))$, $\mbox{dist}(h(\alpha), h^2(\alpha))$ and $\mbox{dist}(h^2(\alpha), h^3(\alpha))$. Finally, let
$\epsilon >3\epsilon '>0$
be such that if dist$(x, y)<3\epsilon '$ then dist$(h(x), h(y))<\epsilon$ and dist$(h^{-1}(x), h^{-1}(y))<\epsilon$.

Our aim is now to find a closed curve $\cal C$ such that $L \subset \hbox{int}{\cal C} \subset U$ and to compute the index of the vector field
$\zeta (x) = h(x) - x$ on $\cal C$ (or the one of $\zeta '(x) = h^{-1}(x)-x$). If it is non zero, we will have reached a contradiction proving the
theorem.

 Let $\hat{L}$ be $L^+=L \bigcup \alpha \bigcup h(\alpha) \bigcup h^2(\alpha)$ plus all the bounded components of $\RR^2 \setminus L^+$. We now
apply
Lemma~\ref{dbislem} to $\hat{L}$ and the arcs $h^{-1}(\gamma)$, $\gamma$ , $h(\gamma)$, $h^2(\gamma)$ and $h^{-1}(\delta)$, $\delta$ ,
$h(\delta)$, $h^2(\delta)$ to get an arc $\eta$ from
$h(\gamma)$ to $h(\delta)$ in $U\setminus \hat{L}$ which is $\epsilon '$-close to $\hat{L}$. Adding subarcs of $h(\gamma)$ and $h(\delta)$ to
$\eta \bigcup h(\alpha)$ we get an oriented simple closed curve $\cal C$. By construction the
fixed point $p_0$ does not belong to int$\cal C$. 

The arc $\eta$ comes equiped with a sequence of
successive points, $r_0 \in h(\gamma), r_1, \ldots, r_n \in h(\delta)$ such that
diam$r_ir_{i+1}<\epsilon'$ for each $i,
0 \leq i \leq n-1$, and for each $i, 1 \leq i \leq n-1$, an arc $\rho_i$
inside $\cal C$, disjoint from all the arcs $h^{-1}(\gamma)$, $\gamma$ , $h(\gamma)$, $h^2(\gamma)$ and $h^{-1}(\delta)$, $\delta$ ,
$h(\delta)$, $h^2(\delta)$, irreducible from $r_i$ to $\hat{L}$, such that diam$\rho_i< \epsilon '$ and therefore
so that each one of the cuts $\rho_i r_i
r_{i+1}\rho_{i+1}$ of ${\bf R}^2 \setminus K$ is
disjoint from its image under $h$ or $h^{-1}$. 

We forget the $\rho_i$ with endpoint on $\alpha$, $h(\alpha)$ or $h^2(\alpha)$ (recall that these three arcs are disjoint). By choice of $\epsilon'$,
we still have cuts disjoint from their
images under $h$or $h^{-1}$. Indeed, if $c$ is a cut of ${\bf R}^2 \setminus K$ subarc of $\alpha$, $h(\alpha)$ or $h^2(\alpha)$, $\rho_{k+1},
\ldots, \rho_{l-1}$ the $\rho_i$ with endpoint on $c$ and $d$ is the cut $\rho_k r_k r_l \rho_l$ obtained by forgetting $\rho_{k+1}, \cdots,
\rho_{l-1}$, then every point of $d$ has distance less than $3\epsilon '$ to $c$, therefore every point of $h(d)$ has distance less than
$\epsilon$ to $h(c)$ and $h(d) \bigcap d = \emptyset$ since dist$(c, h(c))>2\epsilon$.

We will distinguish four cases according to the order of the pairs of prime ends $([\gamma], h([\gamma])) \in a_1$ and $([\delta], h([\delta]))
\in a_2$ on the circle of prime ends.

\subsection{First case} $h([\gamma])$ precedes $[\gamma]$ and $h([\delta])$ precedes $[\delta]$.

First remark that $\gamma \bigcup \alpha \bigcup \delta$ separates int$\cal A$ into two regions
and by hypothesis in this case the parts of $h(\gamma)$ and $h(\delta)$ close to $\cal A$ do not belong to the same region. Therefore $h(\delta)$
has
to meet $\alpha$ before ending in $h(q)$ and $p$ and $q$ are separated by $h(\delta) \bigcup h(\alpha) \bigcup h(\gamma)$. Therefore
$p \in \mbox{\rm int}{\cal C}$ and $q \in \mbox{\rm ext}{\cal C}$ (even $\delta \subset \mbox{\rm ext}{\cal C}$).

Let $l$ be the last point of intersection of $\eta \bigcup r_nh(q)$ and $h^{-1}(\eta)$ (on $r_0h(q) \subset {\cal C}$ oriented from $r_0$ to
$h(q)$),
and $m$ be the first point of intersection of $\alpha$ and $h(\delta)$ on $\alpha$ (oriented from $q$ to $p$). Notice that $l$ precedes $m$ on
$\cal
C$ and that the arc $lm$ on $h^{-1}({\cal C})$ lies outside $\cal C$. Also if $l \in h(\delta)$, then, since $h(l) \in \eta$, $h(l)$ precedes $l$
on $\eta \bigcup r_n h(q)$.

Notice that the intersections $\alpha \bigcap \eta$, $h^{-1}(\eta) \bigcap h(\alpha)$ and $h(\alpha) \bigcap \alpha$ are empty.

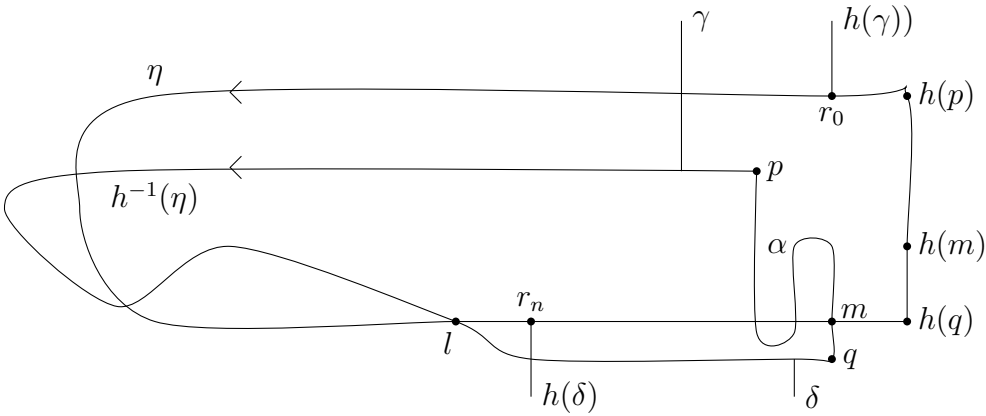
\begin{center}
\begin{tikzpicture}[scale=1] \useasboundingbox (0,0) rectangle (15,5); 
\draw plot[smooth] coordinates{(12,1) (12,2) (12,4) (11,4) (2,4) (1,2.5) (2,1) (6,1) (7,1) (11,1) (12,1)};
\fill (12,1) circle (1.5pt);
\draw (12,1) node[right]{$h(q)$};
\fill (12,2) circle (1.5pt);
\draw (12,2) node[right]{$h(m)$};
\fill (12,4) circle (1.5pt);
\draw (12,4) node[right]{$h(p)$};
\fill (11,4) circle (1.5pt);
\draw (11,4) node[below]{$r_0$};
\fill (6,1) circle (1.5pt);
\draw (5.9,1) node[below]{$l$};
\fill (7,1) circle (1.5pt);
\draw (7,1) node[above]{$r_n$};
\fill (11,1) circle (1.5pt);
\draw (11.3,.9) node[above]{$m$};
\draw plot[smooth] coordinates{(10,3) (1.5,3) (0,2.5) (1.5,1.2) (3,2) (6,1) (7,.5) (10.5,.5) (11,.5) (11,1) (11,2) (10.5,2) (10.5,.85) (10,.85) (10,3)}; 
\fill (10,3) circle (1.5pt);
\draw (10,3) node[right]{$p$};
\fill (11,.5) circle (1.5pt);
\draw (11,.5) node[right]{$q$};
\draw (11,4)--(11,5);
\draw (11,5) node[right]{$h(\gamma)$)};
\draw (9,3)--(9,5) node[right]{$\gamma$};
\draw (10,2) node[right]{$\alpha$};
\draw (2,4) node[above]{$\eta$};
\draw (10.5,.5)--(10.5,0);
\draw (10.5,0) node[right]{$\delta$};
\draw (7,1)--(7,0);
\draw (7,0) node[right]{$h(\delta)$};
\draw (2,3) node[below]{$h^{-1}(\eta)$};
\draw (3.15,4.2)--(3,4.05)--(3.15,3.9);
\draw (3.15,3.2)--(3,3.05)--(3.15,2.9);

\end{tikzpicture}
\end{center}

We will compute the index of $h^{-1}$ along $\cal C$ as the sum of three contributions: the index of $h^{-1}$
along the subarc
$h(p)h(l)$ on $\cal C$, then along the subarc $h(l)h(m)$ and finally along the subarc $h(m)h(p) \subset h(\alpha)$ which we denote by $i_1 $,
$i_2$ and $i_3$ respectively. 

 We will distinguish two subcases according to the position of $h(l)$ which lies before or
after
$l$ on $\eta \bigcup r_nh(q)\subset \cal C$. 

\subsubsection{Subcase 1} $h(l)$ lies after $l$ on $\eta \bigcup r_nh(q)\subset \cal C$ (then, since $h(l) \in \eta$, $l$ and $h(l)\in \eta$): 

Let $k$ be the index such that $h(l)$ lies between $r_k$ and $r_{k+1}$. Using Lemma~\ref{Kbislem} on $K$ and the arc $h(p)h(l) \subset {\cal C}$
subdivided by the points $h(p), r_1, \dots, r_k, h(l)$ with the arcs $\rho_1, \ldots, \rho_k$ we get $i_1 = j_1 + n, n\geq 0$ where $j_1$ is the
index of a vector field whose origin describes
$h(p)h(l)$ while its extremity describes an arc from $p$ to $l$ inside $\cal C$. Indeed, hypothesis 3) of this Lemma is verified, for
if an arc $\rho$ goes from $h(p)h(l)$ towards $K$ then $h^{-1}(\rho)$, which is issued from $h^{-1}(\eta)$ before $l$ must step
into ${\cal C}$
since  $K$ does not meet the components of $\mbox{\rm int}h^{-1}({\cal C}) \bigcap \mbox{\rm ext}{\cal C}$ except perhaps
the one which contains $lm$ in its frontier. Hypothesis 4) too is verified at $h(p)$ since $\rho_1$ lies in the region determined by $\gamma$,
$h(\gamma)$ and $\hat{L}$ by choice of $\epsilon$ and therefore the region $\Omega_0$ determined by $ h(p)r_0r_0r_1\rho_1$ and $\hat{L}$ is
disjoint from its image by $h^{-1}$. It is verified also at $h(l) \in \eta$ since by choice of $\epsilon$ and $\epsilon '$, $l$ lies before $r_k$
on $\eta$, $h^{-1}(r_k)$ precedes $l$ on $h^{-1}(\eta)$ and $h^{-1}(\rho_k)$ lie inside $h^{-1}(\CaC)$.

Lemma~\ref{ilem5} imply that $i_2 = j_2$ where
$j_2$ is the index of a vector field whose origin describes $h(l)h(m)$ while its extremity describes an arc from $l$ to $m$ inside $\cal C$ and
Lemma~\ref{ilem3}
that $i_3 = j_3$ where $j_3$ is the index of a vector field whose origin describes $h(m)h(p)$ while its extremity describes an arc from $m$ to $p$
inside $\cal C$.

The sum $j_1+j_2+j_3$ is equal  to $1$ since it computes the index of a vector field whose origin describes $\cal C$ while its end point stays
inside $\cal C$. Therefore we get that the index of $h^{-1}$ along $\cal C$, which is $i_1+i_2+i_3$, is equal to $1 + n \geq 1$ in contradiction
to the
hypothesis that there is no fixed point for $h^{-1}$ inside $\cal C$.

\subsubsection{Subcase 2} $h(l)$ lies before $l$ on $\eta \bigcup r_nh(q)$: 

For the computation of $i_1$, we can repeat everything said in subcase 1, except for the verification of hypothesis 4 of Lemma~\ref{Kbislem} at $h(l)$.
 But now we want to get $i_1=j_1+n$ for some $n\geq 1$ and we need a more detailed study of the curves $\CaC$ and $h^{-1}(\CaC)$ near $h(l)$ and $l$.

 Notice first that since a cut $c$ subarc of $\alpha$ separates  $h(q)$ from $\infty$, the cut $h(c)$ separates $h^2(q)$ from $\infty$ and there is
a {\it special cut} (that is one of the form $\rho_i r_i r_{i+1} \rho_{i+1}$) which contains $h(c)$ in the bounded region it determines with $K$
 and which separates $h^2(q)$ from $\infty$. If we call $h(d)$
this special cut, then the cut $d$ separates $h(q)$ from $\infty$ and contains $c$ in its associated bounded region. Since $d \bigcap h(d) =
\emptyset$ there are apriori three possibilities for the relative position of $d$ and $h(d)$. But $d$ cannot be contained in the bounded region associated
to $h(d)$ since $h(\delta) \bigcap h(\alpha) = \{ h(q) \}$ and $h(d)$ cannot be contained in the bounded region determined by $d$ since in that case we
would have $l$ before $h(l)$ on $\eta \bigcup r_nh(q)$. 

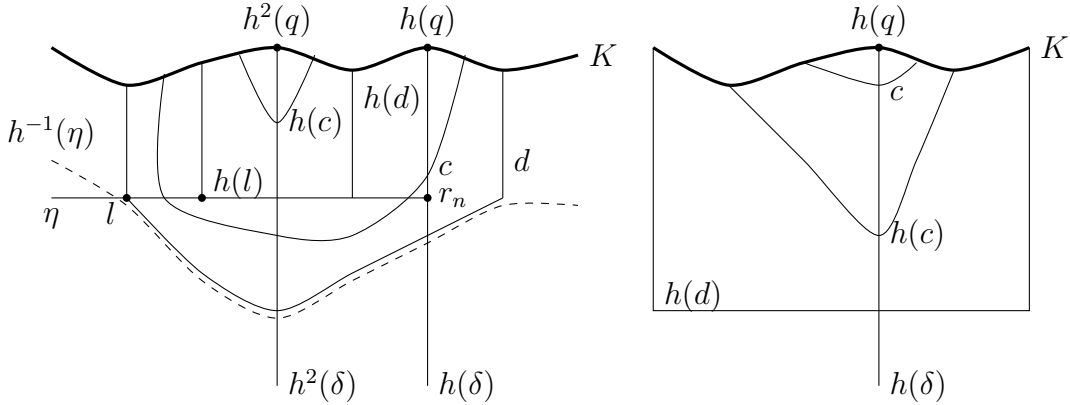
\begin{center}
\begin{tikzpicture}[scale=1] \useasboundingbox (0,0) rectangle (15,5); 
\draw [very thick] plot[smooth] coordinates{(0,4.5) (1,4) (2,4.3) (3,4.5) (4,4.2) (5,4.5) (6,4.2) (7,4.4)};
\draw (7,4.4) node[right]{$K$};
\draw (5,4.5)--(5,0);
\draw (5,0) node[right]{$h(\delta)$};
\fill (5,4.5) circle (1.5pt);
\draw (5,4.5) node[above]{$h(q)$};
\draw (3,4.5)--(3,0);
\draw (3,0) node[right]{$h^2(\delta)$};
\fill (3,4.5) circle (1.5pt);
\draw (3,4.5) node[above]{$h^2(q)$};
\draw (0,2.5)--(5,2.5);
\draw (0,2.5) node[below]{$\eta$};
\fill (5,2.5) circle (1.5pt);
\draw (5,2.5) node[right]{$r_n$};
\fill (1,2.5) circle (1.5pt);
\draw (1,2.3) node[left]{$l$};
\draw (1,2.5)--(1,4);
\draw (6,4.2)--(6,2.5);
\draw (6,3) node[right]{$d$};
\draw plot[smooth] coordinates{(1,2.5) (2,1.5) (3,1) (4,1.5) (5,2) (6,2.5)};
\draw (2,4.3)--(2,2.5);
\draw (4,4.2)--(4,2.5);
\draw (4,3.8) node[right]{$h(d)$};
\fill (2,2.5) circle (1.5pt);
\draw (2,2.7) node[right]{$h(l)$};
\draw plot[smooth] coordinates{(2.5,4.4) (3,3.5) (3.5,4.4)};
\draw (3,3.5) node[right]{$h(c)$};
\draw plot[smooth] coordinates{(1.5,4.15) (1.5,2.5) (3,2) (4,2) (5,2.8) (5.5,4.4)};
\draw (5,2.9) node[right]{$c$};
\draw [dashed] plot[smooth] coordinates{(0,3) (1,2.4) (2,1.4) (3,.9) (4,1.4) (5,1.9) (6,2.4) (7,2.4)};
\draw (0,3) node[above]{$h^{-1}(\eta)$};
\draw [very thick] plot[smooth] coordinates{(8,4.5) (9,4) (10,4.3) (11,4.5) (12,4.2) (13,4.5)};
\draw (13,4.5) node[right]{$K$};
\draw (11,4.5)--(11,0);
\draw (11,0) node[right]{$h(\delta)$};
\fill (11,4.5) circle (1.5pt);
\draw (11,4.5) node[above]{$h(q)$};
\draw (8,4.5)--(8,1)--(13,1)--(13,4.5);
\draw (8,1.2) node[right]{$h(d)$};
\draw plot[smooth] coordinates {(9,4) (10,3) (11,2) (11.5,3) (12,4.2)};
\draw plot[smooth] coordinates {(10,4.3) (11,4) (11.5,4.3)};
\draw (11,3.9) node[right]{$c$};
\draw (11,2) node[right]{$h(c)$};
\end{tikzpicture}
\end{center}

Therefore the bounded
regions associated to $d$ and $h(d)$ are disjoint and $l$ is the
last point on $d$ (starting from the endpoint of $d$ before $h(\delta)$ on the circle of prime ends) of $\eta \bigcup r_nh(q)$ : between $l$ and $m$, $\eta
\bigcup h(\alpha)$ and $h^{-1}(\eta) \bigcup \alpha$ are disjoint.

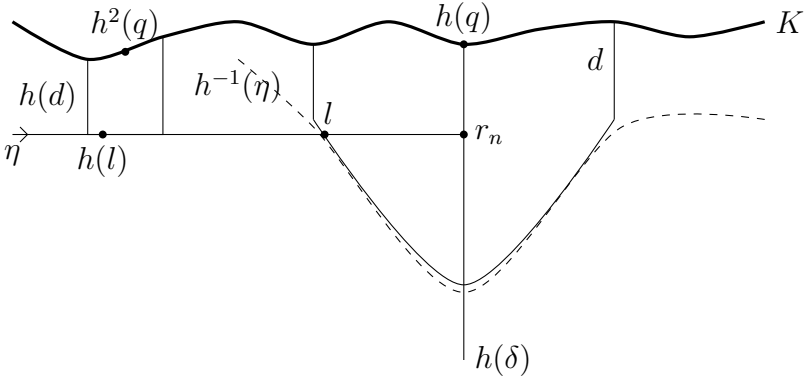
\begin{center}
\begin{tikzpicture}[scale=1] \useasboundingbox (0,0) rectangle (15,5); 
\draw [very thick] plot[smooth] coordinates{(0,4.5) (1,4) (2,4.3) (3,4.5) (4,4.2) (5,4.5) (6,4.2) (7,4.4) (8,4.5) (9,4.3) (10,4.5)};
\draw (10,4.5) node[right]{$K$};
\draw (1,4)--(1,3)--(2,3)--(2,4.3);
\fill (1.5,4.1) circle (1.5pt);
\draw (1.5,4.1) node[above]{$h^2(q)$};
\draw (1,3.5) node[left]{$h(d)$};
\draw (4,4.2)--(4,3.2);
\draw (8,4.5)--(8,3.2);
\draw (6,4.2)--(6,0);
\draw (6,0) node[right]{$h(\delta)$};
\draw plot[smooth] coordinates{(4,3.2) (6,1) (8,3.2)};
\draw (0,3)--(6,3);
\draw (6,4.2) node[above]{$h(q)$};
\fill (6,3) circle (1.5pt);
\draw (6,3) node[right]{$r_n$};
\fill (6,4.2) circle (1.5pt);
\fill (1.2,3) circle (1.5pt);
\draw (1.2,3) node[below]{$h(l)$};
\fill (4.15,3) circle (1.5pt);
\draw (4.2,3) node[above]{$l$};
\draw [dashed] plot[smooth] coordinates{ (3,4) (4,3.1) (6,.9) (8.1,3.1) (10,3.2)};
\draw (3,4) node[below]{$h^{-1}(\eta)$};
\draw (8,4) node[left]{$d$};
\draw (0,3) node[below]{$\eta$};
\draw (.1,2.9)--(.2,3)--(.1,3.1);

\end{tikzpicture}
\end{center}

Therefore, between $h(\gamma)$ and $h^2(\delta)$, there exists a special cut $\tilde c$ which contains an end point of $a_1$ (invariant subarc of
the circle of prime ends) that is a fixed point on the circle of prime ends. Let $\tilde \Omega$ be the bounded region determined by  $\tilde c$
and $K$. If $\tilde \Omega \subset h^{-1}(\tilde \Omega)$ we get $n \geq 1$ (see Lemma~\ref{Kbislem}). If $ h^{-1}(\tilde \Omega) \subset
\tilde \Omega$, then that endpoint of $a_1$ is a repulsor for the map $\hat{h}$ induced by $h$ on the circle of prime ends and the position of $\delta$ and $h(\delta)$
(or $\gamma$ and $h(\gamma)$) on that circle imply that there is another fixed point between $[h(\gamma)]$ and $[h^2(\delta)]$ which is an attractor for
$\hat{h}$. This gives a special cut $\hat{c}$ with $\hat{\Omega} \subset h^{-1}(\hat{\Omega})$ so that in any case, $n\geq 1$.

Also $i_2=j_2 - 1$ by Lemma~\ref{ilem5}, $i_3 = j_3$ by Lemma~\ref{ilem3} and $j_1+j_2+j_3=1$. We conclude again that  $i_1+i_2+i_3$ is  equal
to $1+n-1 \geq 1$ to get the same contradiction.

\begin{rem} In the situation of subcase 1, $h(d)$ is in the bounded region determined by $d$ and one gets also $n \geq 1$ in that subcase (see the
picture), but
this information was not necessary there.
\end{rem}

\subsection{Second case} $h([\gamma])$ follows $[\gamma]$ and $[h(\delta)]$ precedes $[\delta]$.

We will compute the index of $h$ along the curve $h^{-1}(\cal C)$. 

Since $h(\gamma) \bigcap \gamma \bigcup \alpha \bigcup \delta = \emptyset$, one has $h(p) \in \mbox{\rm int}h^{-1}(\cal C)$. Also $h(q) \in
\mbox{\rm int}h^{-1}(\cal C)$, otherwise, since $h(\alpha) \bigcap (h^{-1}(\eta) \bigcup \alpha \bigcup \gamma) = \emptyset$, we would have that
$h(\alpha)$ cuts $\delta$ and so $h(\delta)$ would cut $\delta$.

Again, we will distinguish two subcases.

\subsubsection{Subcase 1} $h(\delta)$ cuts $\alpha$. And therefore $h(\alpha) \bigcap \delta = \emptyset$ and $h(\alpha) \subset \mbox{\rm
int}h^{-1}(\cal C)$.

\begin{center}
\begin{tikzpicture}[scale=1] \useasboundingbox (0,0) rectangle (15,5); 
\draw plot[smooth] coordinates{(12,1) (12,2) (12,4) (11,4) (2,4) (1,2.5) (2,1) (6,1) (7,1) (11,1) (12,1)};
\fill (12,1) circle (1.5pt);
\draw (12,1) node[right]{$q$};
\fill (12,3) circle (1.5pt);
\draw (12,3.1) node[left]{$h(n)$};
\fill (12,2) circle (1.5pt);
\draw (12,1.8) node[right]{$h(m)$};
\fill (12,4) circle (1.5pt);
\draw (12,4) node[right]{$p$};
\fill (11,1) circle (1.5pt);
\draw (11.15,1) node[above]{$m$};
\draw plot[smooth] coordinates{(10,3) (1.5,3) (0,2.5) (1.5,1.2) (3,2) (6,1) (7,.5) (9,.5)};
\draw plot[smooth] coordinates{(9,1) (12,2) (13,3.5) (12,3) (10,2) (10,3)}; 
\draw (9,0)--(9,1);
\fill (10,3) circle (1.5pt);
\draw (10,3) node[right]{$h(p)$};
\fill (10,2) circle (1.5pt);
\draw (10,2) node[left]{$h(q)$};
\draw (2,4) node[above]{$h^{-1}(\eta)$};
\draw (9,.5)--(9,0);
\draw (10.5,0) node[right]{$\delta$};
\draw (10.5,0)--(10.5,1);
\draw (9,0) node[right]{$h(\delta)$};
\draw (2,3) node[below]{$\eta$};
\draw (3.15,4.2)--(3,4.05)--(3.15,3.9);

\end{tikzpicture}
\end{center}

\begin{center}
\begin{tikzpicture}[scale=1] \useasboundingbox (0,0) rectangle (15,5); 
\draw [very thick] plot[smooth] coordinates{(1,1) (2,.5) (3,1) (4,1.5) (5,1) (6,1.5)};
\draw (6,1.5) node[right]{$K$};
\draw (5,1)--(5,4);
\fill (5,1) circle (1.5pt);
\draw (5,1) node[below]{$p$};
\draw (5,4) node[right]{$\gamma$};
\draw (3,1)--(3,4);
\fill (3,1) circle (1.5pt);
\draw (3.1,1) node[below]{$h(p)$};
\draw (3,4) node[right]{$h(\gamma)$};
\draw (1,3)--(5,3);
\fill (5,3) circle (1.5pt);
\draw (5,3) node[right]{$h^{-1}(r_0)$};
\draw [out=110,in=-100] (4,1.5) to (4,3);
\draw [out=110,in=-80] (2,.5) to (2,3.5);
\draw plot[smooth] coordinates{(1,3.7) (2,3.5) (3,2.7)};
\fill (3,2.7) circle (1.5pt);
\draw (3,2.7) node[right]{$r_0$};
\draw (1,3.7) node[left]{$\eta$};
\draw (2,2.5) node[left]{$\rho_1$};
\draw (4,3) node[above]{$h^{-1}(\rho_1)$};

\end{tikzpicture}
\end{center}

Let $h(m)$ (resp. $h(n)$) denote the first (resp. last)
intersection point of $h(\delta)$ and $\alpha$ on $\alpha$ oriented from $q$ to $p$.

We compute our index as the sum $i_1 + i_2$ where $i_1$ is the index of $h$ along the subarc $pm$ of $h^{-1}(\cal C)$ and $i_2$ the index of $h$
along the subarc $mp$ of $h^{-1}(\cal C)$. 

The arc $pm$ comes equipped with the points $h^{-1}(r_i)$ and the arcs $h^{-1}(\rho_i)$ which gives a sequence of special cuts on $pm$ disjoint
from their images under $h$. To apply Lemma~\ref{Kbislem}, we first verify its third hypothesis. If an arc $\rho$ goes
from $pm \subset h^{-1}(\cal C)$ to $K$, then $h(\rho)$ from $h(p)h(m)$ towards $K$ must go inside $h^{-1}(\cal C)$ since $K$
 does not meet the components of $\mbox{\rm ext}h^{-1}({\cal C}) \bigcap \mbox{\rm int}{\cal C}$. As for the fourth
hypothesis, note that since $h^{-1}(\rho_1) \bigcap h(\gamma) = \emptyset$, the
bounded region determined by $ph^{-1}(r_0)h^{-1}(r_1)h^{-1}(\rho_1)$ is contained in the the region between $\gamma$ and $h(\gamma)$ and does not
meet its image under $h$. At the other end of the arc $pm$, we have $ h^{-1}(\rho_{n-1}) \bigcap h(\delta) = \emptyset$ and so
$h^{-1}(\rho_{n-1})h^{-1}(r_{n-1})h^{-1}(r_n)m$ is contained in the region bounded by $\delta$ and $h(\delta)$, whence
hypothesis 4). Lemma~\ref{Kbislem}, which can now be
applied gives then $i_1 =j_1 +n$, $n\geq 0$, where $j_1$ be the index along $pm$ of a vector
field whose origin describes $pm$ while its endpoint describes a path inside $h^{-1}(\cal C)$ from $h(p)$ to $h(m)$. 

Let $j_2$ be the index along $mp$ of a vector field along $mp$ 
whose origin describes $mp$ while its endpoint follows the curve obtained by replacing in $h(m)h(p) \subset \cal C$ the subarc  $h(m)h(n)
\subset \cal C$  by the subarc of $\alpha$ with the same endpoints.

One has $i_2 = j_2$ since the subarcs from $h(m)$ to $h(n)$ on $\cal C$ and
$\alpha$ are homotopic rel their endpoints in $\RR^2 \setminus \delta$ (and $mn \subset \delta$)

Since $j_1 +j_2 =1$ (Lemma~\ref{ilem3}), we get $i_1 + i_2 = 1+n \geq 1$, a contradiction.

\subsubsection{Subcase 2} $h(\delta)$ does not cut $\alpha$

\begin{center}
\begin{tikzpicture}[scale=1] \useasboundingbox (0,0) rectangle (18,5); 
\draw plot[smooth] coordinates{(5,1) (5,4) (2,4) (1,1) (2,1) (5,1)};
\fill (5,1) circle (1.5pt);
\draw (5,1) node [right]{$p$};
\fill (5,4) circle (1.5pt);
\draw (5,4) node[right]{$q$};
\draw plot[smooth] coordinates{(3,3) ((1,3.5) (0,2.5) (1,2.5) (1.5,2.5) (1.5,1.5) (1.5,.5) (2,.5) (2,1) (3,2) (5,.5) (3,3)};
\draw (4,1)--(4,0);
\draw (4,0) node[right]{$\delta$};
\fill (4,1) circle (1.5pt);
\fill (3,3) circle (1.5pt);
\draw (3,3) node[right]{$h(p)$};
\fill (3,2) circle (1.5pt);
\draw (3,2) node[above]{$h(q)$};
\draw (2.5,1.75)--(2.5,0);
\draw (2.5,0) node[right]{$h(\delta)$};
\fill (2.5,1.75) circle (1.5pt);

\draw [very thick] plot[smooth] coordinates{(8,4) (9,3.5) (10,4) (11,3.8) (12,4)};
\fill (12,4) circle (1.5pt);
\draw (12,3.8) node[right]{$q$};
\fill (10,4) circle (1.5pt);
\draw (10,4) node[above]{$h(q)$};
\draw (8,1)--(12,1);
\draw (8,1) node[below]{$h^{-1}(\eta)$};
\fill (12,1) circle (1.5pt);
\draw (12,1) node[right]{$h^{-1}(r_n)$};
\draw (12,0)--(12,4);
\draw (12,3) node[right]{$\delta$};
\draw (11,1)--(11,3.8);
\draw (11.2,3.8) node[above]{$h^{-1}(\rho_{n-1})$};
\draw (10,0)--(10,4);
\draw (10,0) node[left]{$h(\delta)$};
\fill (10,2) circle (1.5pt);
\draw (10,2) node[right]{$r_n$};
\draw (8,2)--(10,2);
\draw (8,2) node[below]{$\eta$};
\draw (9,2)--(9,3.5);
\draw (9,3) node[left]{$\rho_{n-1}$};

\end{tikzpicture}
\end{center}

Let $i_1$ be the index of $h$ along $pq \subset h^{-1}(\cal C)$. The arc $pq$ is again equipped with the points $h^{-1}(r_i)$ and the arcs
$h^{-1}(\rho_i)$ which gives a sequence of special cuts on $pm$ disjoint
from their images under $h$.
Hypothesis 3. and 4. of Lemma~\ref{Kbislem} are verified as in subcase 1 above except for hypothesis 4. at $q$ where we use that
$h^{-1}(\rho_{n-1}) \bigcap h(\alpha) = \emptyset$ to show that the region determined by $h^{-1}(\rho_{n-1})h^{-1}(r_n)q$ is disjoint from its
image. Therefore, there exists
 a vector field
whose origin describes $pq$ while its endpoint describes first an arc inside $h^{-1}(\cal C)$ from $h(p)$ to $h(q)$ 
and whose index $j_1$ satisfies
 $i_1 = j_1 +n, n\geq 0$. 

Since $h(\alpha)$ is homotopic rel endpoints in $\RR^2 \setminus \gamma \bigcup \alpha$ to an arc inside $h^{-1}(\cal C)$, the index $i_2$ of $h$
along $qp$ is equal to the index $j_2$ of a vector field
whose origin describes $qp$ while its endpoint describes an arc inside $h^{-1}(\cal C)$.

Since, by Lemma~\ref{ilem3} $j_1 + j_2=1$, we have again a contradiction.

\subsection{Third case} $h([\gamma])$ follows $[\gamma]$ and $h([\delta])$ follows $[\delta]$.

This case reduces to the first one by exchanging $h$ and $h^{-1}$.

\subsection{Fourth case} $h([\gamma])$ precedes $[\gamma]$ and $h([\delta])$ follows $[\delta]$.

This case reduces to the second one by exchanging $h$ and $h^{-1}$.

\newcommand{\etalchar}[1]{$^{#1}$}

\end{document}